\documentclass{article}
\usepackage{iclr2027_conference,times}
\iclrfinalcopy

\usepackage{amsmath,amsfonts,bm}

\def\eqref#1{equation~(\ref{#1})}

\def\plaineqref#1{(\ref{#1})}

\def\1{\bm{1}}

\DeclareMathAlphabet{\mathsfit}{\encodingdefault}{\sfdefault}{m}{sl}
\SetMathAlphabet{\mathsfit}{bold}{\encodingdefault}{\sfdefault}{bx}{n}

\newcommand{\E}{\mathbb{E}}

\newcommand{\R}{\mathbb{R}}

\usepackage{amsmath,amssymb,amsthm,mathtools,bm}
\usepackage{booktabs}
\usepackage{colortbl}
\usepackage{graphicx}
\usepackage{capt-of}
\usepackage{microtype}
\usepackage{algorithm}
\usepackage{algpseudocode}
\usepackage{xspace}
\usepackage{hyperref}
\usepackage{url}
\hypersetup{colorlinks=true,linkcolor=black,citecolor=black,urlcolor=blue}

\newtheorem{theorem}{Theorem}[section]
\newtheorem{lemma}[theorem]{Lemma}

\newtheorem{corollary}[theorem]{Corollary}
\newtheorem{assumption}[theorem]{Assumption}

\newcommand{\St}{\operatorname{St}}
\newcommand{\grad}{\operatorname{grad}}
\newcommand{\sym}{\operatorname{sym}}

\newcommand{\one}{\mathbf 1}
\newcommand{\Xb}{\mathbf X}
\newcommand{\Sb}{\mathbf S}
\newcommand{\Gb}{\mathbf G}
\newcommand{\Wb}{\mathbf W}
\newcommand{\Jb}{\mathbf J}
\newcommand{\Cman}{\mathcal C}
\newcommand{\ip}[2]{\left\langle #1,#2\right\rangle}
\newcommand{\norm}[1]{\left\|#1\right\|}
\newcommand{\plados}{\textsc{Plados}\xspace}

\title{\textsc{PLADOS}: A Parameter-free Landing Algorithm for
Decentralized Optimization on the Stiefel Manifold}
\author{Shu Li\\
Qiuzhen College\\
Tsinghua University\\
Beijing, China\\
\texttt{lishu26@mails.tsinghua.edu.cn}
\And
Jiang Hu\thanks{Corresponding author.}\\
Yau Mathematical Sciences Center\\
Tsinghua University\\
Beijing, China\\
\texttt{jianghu@tsinghua.edu.cn}}

\begin{document}
\maketitle
\lhead{}
\addtocontents{toc}{\protect\setcounter{tocdepth}{-1}}

\begin{abstract}
Decentralized optimization on the Stiefel manifold has broad applications
in machine learning and signal processing.
We propose \plados, a parameter-free landing algorithm for decentralized
optimization on the Stiefel manifold. Each iteration combines local
stochastic gradient updates with a single retraction-free communication step over
the communication graph. A key convex-like property of the intersection of the nonconvex Stiefel manifold and
consensus constraints is characterized through a restricted secant
inequality for the penalty-only scheme, ensuring local contraction of the manifold consensus error without introducing
additional parameters that require tuning. For \(n\) nodes using independent local batches of size \(b\),
we prove that \plados{} achieves linear speedup with respect to the number
of nodes with an asymptotic convergence rate of
\(O(\sigma/\sqrt{nbK}+(n\xi^2)^{1/3}/K^{2/3})\)
after \(K\) iterations, where \(\sigma^2>0\) is the sampling variance
and \(\xi^2\) quantifies gradient heterogeneity.
Experiments demonstrate stability across tested stepsizes and networks, time efficiency,
and linear speedup with respect to the number of nodes.
\end{abstract}

\section{Introduction}
\label{sec:introduction}

Decentralized optimization enables a network of \(n\) nodes to learn from
local data using only neighbor communication
\citep{tsitsiklis1986distributed,nedic2009distributed,duchi2011dual}.
Distributing both data and computation makes this framework useful for
large-scale learning and signal processing. Many such applications also
require orthonormal decision matrices, including principal component
analysis \citep{Chen2021DecentralizedRG}, low-rank estimation
\citep{Deng2025DecenProjRG}, orthogonal representation learning
\citep{Huang2018OWN}, and low-rank adaptation \citep{Zhang2026Landing}.
These problems motivate decentralized stochastic optimization on the
Stiefel manifold \citep[Section~7.3]{Boumal2023Manifolds}:
\begin{equation}
 \min_{X\in\St(p,r)} f(X):=\frac1n\sum_{i=1}^n f_i(X),
 \qquad \St(p,r):=\{X\in\R^{p\times r}:X^\top X=I_r\},
 \label{eq:intro-problem}
\end{equation}
where node \(i\) accesses its local objective through a stochastic oracle
and communicates with its neighbors over a connected graph.

Solving \eqref{eq:intro-problem} requires the nodes to reach consensus
while satisfying a nonconvex orthogonality constraint. Even when all local
matrices have orthonormal columns, their weighted average can leave the
Stiefel manifold. Retractions restore feasibility after each update, but
the required QR or polar factorizations add computation at every node
\citep{Chen2021DecentralizedRG,Deng2025DecenProjRG}. Retraction-free landing methods offer a cheaper way to reduce constraint violation
\citep{ablin2022fast,Zhang2026Landing}. In a decentralized algorithm,
however, landing must work jointly with~communication~and~optimization.

Stochastic gradients and heterogeneous data make this interaction more
delicate. Independent local samples reduce the variance of the averaged
gradient, while differences between local objectives drive the nodes
apart. On the Stiefel manifold, disagreement also affects feasibility and
the descent of the projected network average. A convergence analysis must
therefore control these errors together while preserving the variance
reduction from independent local batches. This is essential for obtaining
linear speedup with respect to the number of nodes, a central benefit
of decentralized stochastic optimization \citep{Lian2017DPSGD,Tang2018D2}.

This raises the question: can a simple retraction-free decentralized
algorithm control both consensus and orthogonality errors and achieve
stochastic linear speedup with respect to the number of nodes? We answer this question
with \plados, a parameter-free landing algorithm that combines a
penalty-only scheme with one communication round and one local stochastic
gradient step per iteration.
To the best of our knowledge, \plados is the first decentralized stochastic
landing method with proven convergence on the Stiefel manifold.

\subsection{Related Work}
\label{subsec:related-work}

\paragraph{Decentralized optimization in Euclidean space.}
Classical consensus and gradient methods laid the foundations
\citep{tsitsiklis1986distributed,nedic2009distributed,duchi2011dual,
yuan2016convergence}. D-PSGD and subsequent analyses establish nonconvex
stochastic linear speedup and characterize topology, compression, and
heterogeneity effects
\citep{Lian2017DPSGD,Koloskova2019DecentralizedSO,Tang2018D2,
alghunaim2022unified,Takezawa2026Topology}.  EXTRA, exact diffusion, NIDS,
and gradient tracking use auxiliary information for exact convergence or
reduced local drift
\citep{shi2015extra,yuan2018exact,Li2019NIDS,Nedic2016AchievingGC,
Qu2017HarnessingSA,Li2019SDIGingAS,pu2021distributed,
koloskova2021improved,xin2021improved,Liu2023DecentralizedGT}.
In D-PSGD, heterogeneity contributes a \(K^{-2/3}\) term, while independent
batches reduce stochastic error \citep{alghunaim2022unified}.
By contrast, \(D^2\) uses a history correction to remove dependence on global
gradient heterogeneity \citep{Tang2018D2}.
Extending these Euclidean techniques to nonconvex manifold constraints is
nontrivial.

\paragraph{Retraction-based decentralized optimization on manifolds.}
Early work studied consensus intrinsically on manifolds
\citep{Sarlette2009ConsensusManifolds,Tron2012RiemannianConsensus}.  On the
Stiefel manifold, DRSGD combines stochastic Riemannian steps with retractions
and multi-step consensus, while DRGTA adds gradient tracking
\citep{Chen2021DecentralizedRG}.  DPRGD, DPRGT, and DPRSRM extend projected
Riemannian techniques to compact embedded submanifolds and stochastic
recursive momentum
\citep{Deng2025DecenProjRG,Deng2025DecenProjRSG}.  Proximal tracking,
Riemannian EXTRA, local-consensus, conjugate-gradient, natural-gradient,
compressed, and quantized variants further broaden this class
\citep{Wang2025DecenPGTRiem,Wu2025RiemannianEC,Hu2025LocalConsensus,
Chen2023DRCG,Hu2025DRNG,Hu2024SingleStepCompression,Chen2025QuantizedRGT}.
Repeated factorizations and vector transport can limit scalability,
especially on GPUs \citep{Song2025EFLanding,Zhang2026Landing}.

\paragraph{Retraction-free optimization on the Stiefel manifold.}
Centralized infeasible, landing, and constraint-dissolving methods avoid
orthonormalization
\citep{gao2019parallelizable,ablin2022fast,Xiao2024DissolvingConstraints,
Zhang2026Landing}. Centrally coordinated EF-Landing combines stochastic
landing with compressed communication \citep{Song2025EFLanding}.
Decentralized examples include the approximate
augmented-Lagrangian method DESTINY, landing with gradient tracking in DRFGT,
and the deterministic primal--dual RF-EXTRA correction
\citep{Wang2022StiefelAL,Sun2024GlobalDecenRFStiefel,Li2026RFEXTRA}.
Stochastic linear speedup with respect to the number of nodes without
gradient tracking remains a key challenge for decentralized methods with
independent mini-batches and heterogeneous objectives.
Table~\ref{tab:method-comparison} summarizes the main properties of
representative methods and \plados.
Penalty selection is another central issue in constrained optimization:
\citet{Byrd1998ActiveSet} asked at the International Congress of Mathematicians 1998 whether sensitivity to penalty
parameters could limit large-scale exact penalty methods. Our \plados answers this question for the considered constrained problem by a landing correction that requires no problem-dependent penalty tuning.

\begin{table}[t]
\caption{Comparison with representative stochastic decentralized and retraction-free
methods.}
\label{tab:method-comparison}
\centering
\scriptsize
\setlength{\tabcolsep}{3pt}
\renewcommand{\arraystretch}{1.15}
\resizebox{\textwidth}{!}{%
\begin{tabular}{lcccccc}
\toprule
\textbf{Method} & \textbf{Decentralized} &
\shortstack{\textbf{Stiefel manifold}\\\textbf{constraints}} &
\shortstack{\textbf{Retraction-}\\\textbf{free}} &
\shortstack{\textbf{Parameter-}\\\textbf{free}} &
\shortstack{\textbf{Heterogeneity}\\\textbf{assumption}} &
\shortstack{\textbf{Linear node}\\\textbf{speedup}} \\
\midrule
EF-Landing~\citep{Song2025EFLanding} & $\times$ & $\checkmark$ & $\checkmark$ & $\times$ & $\times$ & $\checkmark$ \\
Man-LoRA~\citep{Zhang2026Landing} & $\times$ & $\checkmark$ & $\checkmark$ & $\checkmark$ & $\times$ & $\times$ \\
DRSGD~\citep{Chen2021DecentralizedRG} & $\checkmark$ & $\checkmark$ & $\times$ & $\times$ & $\times$ & $\checkmark$ \\
\midrule
\rowcolor{black!7}
\textbf{\plados} & $\checkmark$ & $\checkmark$ & $\checkmark$ & $\checkmark$ & $\checkmark$ & $\checkmark$ \\
\bottomrule
\end{tabular}%
}
\vspace{2pt}

\begin{minipage}{\textwidth}
\scriptsize
\emph{Note.} ``Parameter-free'' means that constraint-related stepsizes or
parameters have explicit values and require no tuning; \plados uses \(\gamma=1/3\).
``Heterogeneity assumption'' records an explicit bound on cross-worker gradient
heterogeneity, not whether heterogeneous data are allowed.
\end{minipage}
\end{table}

\subsection{Contributions}
\label{subsec:contributions}

Our contributions are threefold.
\begin{itemize}
 \item \textbf{A parameter-free landing algorithm.}
 We propose \plados for decentralized optimization on the Stiefel manifold.
 Each iteration combines local stochastic gradient updates with a single
 retraction-free communication step. The landing coefficient is universally
 fixed at \(\gamma=1/3\), independently of unknown problem constants,
 so no additional feasibility parameter requires tuning.

 \item \textbf{Linear convergence of the penalty-only scheme.}
 The feasible set intersects the convex consensus subspace with the
 nonconvex product Stiefel manifold. We establish a restricted secant
 inequality (RSI) for the penalty-only scheme near this intersection,
 yielding local Q-linear convergence of the distance to the feasible set.

 \item \textbf{Linear speedup with respect to the number of nodes.}
 With independent local mini-batches, we prove linear speedup with respect
 to the number of nodes with heterogeneous local data, once
 the stochastic term dominates the heterogeneity and network transients.
 The asymptotic rate is
 \(O(\sigma/\sqrt{nbK}+(n\xi^2)^{1/3}/K^{2/3})\) after \(K\) iterations,
 where \(\sigma^2>0\) is the sampling variance and \(\xi^2\) quantifies
 gradient heterogeneity. This holds for a fixed network or a uniformly
 gapped network family with uniformly bounded \mbox{problem and oracle constants}.
\end{itemize}

\section{Problem Formulation and \plados}
\label{sec:method}

We study decentralized optimization on the Stiefel manifold in
\eqref{eq:intro-problem}.  For a stack
\(\Xb=(X_1,\ldots,X_n)\), define
\[
 \widehat X=\frac1n\sum_iX_i,\qquad
 \bar X=P_{\St}(\widehat X),\qquad
 \bar\Xb=\one\otimes\bar X,
 \qquad d(\Xb)=\norm{\Xb-\bar\Xb}_F.
\]
Thus \(d(\Xb)\) is the distance to the consensus Stiefel manifold
\(\Cman=\{\one\otimes Q:Q\in\St(p,r)\}\).  The projection
\(P_{\St}(Y)=Y(Y^\top Y)^{-1/2}\), defined for full-column-rank \(Y\),
is the nearest-point polar projection \citep{Higham1986Polar}.
For any ambient matrix \(U\), write
\(P_XU=U-X\sym(X^\top U)\), where
\(\sym(A)=(A+A^\top)/2\) for a square matrix \(A\).
At a feasible \(X\), this is the tangent
projection and \(\grad f(X)=P_X\nabla f(X)\)
\citep[Section~7.3]{Boumal2023Manifolds}.

\subsection{Assumptions}

We state the assumptions used in the analysis.
Let \(\mathcal F_k\) denote the history immediately before the time-\(k\)
mini-batches are sampled.
The initialization \(\Xb_0\) is deterministic; its local tube condition is
stated in Theorem~\ref{thm:master}.

\begin{assumption}[Objective regularity and stochastic oracle]
\label{ass:regularity}\label{ass:main}
Each \(f_i\) has a \(C^2\) extension to a neighborhood of a fixed compact
tube around \(\St(p,r)\).  Before iteration \(k\), node \(i\) draws a fresh
mini-batch
\(\mathcal B_{i,k}=\{\zeta_{i,k}^{(1)},\ldots,\zeta_{i,k}^{(b)}\}\)
and forms
\begin{equation}
 S_{i,k}=P_{X_{i,k}}\!\left(
 \frac1b\sum_{\ell=1}^b
 \nabla_X F_i\!\left(X_{i,k};\zeta_{i,k}^{(\ell)}\right)\right),
 \label{eq:minibatch-gradient}
\end{equation}
where \(F_i(\cdot;\zeta)\) is a local sample loss.  Conditional on
\(\mathcal F_k\), the nodewise estimators are independent across nodes and
satisfy
\begin{equation}
 \E\!\left[\frac1b\sum_{\ell=1}^b
 \nabla_XF_i(X_{i,k};\zeta_{i,k}^{(\ell)})
 \,\middle|\,\mathcal F_k\right]=\nabla f_i(X_{i,k}).
\end{equation}
The batch gradient satisfies
\(\|b^{-1}\sum_{\ell=1}^b
\nabla_XF_i(X_{i,k};\zeta_{i,k}^{(\ell)})\|_F\le G\)
almost surely in the tube. Define
\(G_{i,k}:=P_{X_{i,k}}\nabla f_i(X_{i,k})\).
Linearity of \(P_{X_{i,k}}\) gives the first equality below, and we assume
the projected variance bound in the second:
\begin{equation}
 \E[S_{i,k}\mid\mathcal F_k]=G_{i,k},
 \qquad
 \E[\norm{S_{i,k}-G_{i,k}}_F^2\mid\mathcal F_k]\le\frac{\sigma^2}{b}.
 \label{eq:projected-oracle}
\end{equation}
\end{assumption}

The oracle conditions separate the mean gradient from sampling noise.
Independent local batches reduce the variance of the averaged noise to
\(\sigma^2/(nb)\). The almost-sure bound controls tube invariance and
higher-order remainders; smooth finite-sum losses satisfy it on a compact
neighborhood.

\begin{assumption}[Data heterogeneity]
\label{ass:heterogeneity}
For every feasible \(X\),
\begin{equation}
 \frac1n\sum_{i=1}^n\norm{\nabla f_i(X)-\nabla f(X)}_F^2
 \le\xi^2.                                                   
\end{equation}
\end{assumption}

The quantity \(\xi^2\) measures variation among local gradients evaluated
at the same feasible matrix. Smoothness separately controls the effect of
evaluating those gradients at different local iterates.

\begin{assumption}[Communication network]
\label{ass:network}
The nodes communicate through a symmetric doubly stochastic matrix \(W\)
whose graph is connected.  With \(J=\one\one^\top/n\), we have
\(\sigma_2(W):=\norm{W-J}_2<1\); write \(\Wb=W\otimes I_p\) and
\(\Jb=J\otimes I_p\).
\end{assumption}

The spectral gap \(1-\sigma_2(W)\) determines how effectively communication
reduces disagreement between nodes.
These are standard regularity, oracle, heterogeneity, and connectivity
assumptions in decentralized stochastic optimization
\citep{Chen2021DecentralizedRG,Tang2018D2}.

\subsection{The \plados algorithm}

For a coefficient \(\gamma\in(0,1)\), define the landing map
\begin{equation}
 \phi_\gamma(X)=X-\gamma X(X^\top X-I_r).
\end{equation}
Throughout the paper we fix \(\gamma=1/3\), write \(\phi=\phi_{1/3}\), and
set \(\Phi(\Xb)=(\phi(X_1),\ldots,\phi(X_n))\). This universally prescribed
landing coefficient gives the parameter-free feasibility correction.
Section~\ref{sec:theory} proves that it contracts the normal feasibility
error while leaving every feasible point unchanged.

\begin{algorithm}[H]
\caption{\plados}
\label{alg:plods}
\begin{algorithmic}[1]
\Statex \textbf{Input:} communication weights \(w_{ij}\), batch size \(b\),
initial point \(X_{i,0}\), horizon \(K\), and either \(\alpha_0\) or \(\alpha_K\)
\Statex \textbf{Stepsize:} either
\(\alpha_k=\alpha_0/\sqrt{k+1}\) with \(\alpha_0\le\bar\alpha\), or, for a
known horizon, a constant stepsize \(\alpha_k\equiv\alpha_K\le\bar\alpha\)
\For{\(k=0,\ldots,K-1\)}
 \State draw \(\mathcal B_{i,k}=\{\zeta_{i,k}^{(\ell)}\}_{\ell=1}^b\)
 \State compute \(S_{i,k}\) by \eqref{eq:minibatch-gradient}
 \State \(Y_{i,k}=\phi(X_{i,k})
 =X_{i,k}-\frac13X_{i,k}(X_{i,k}^{\top}X_{i,k}-I_r)\)
 \State communicate \(Y_{i,k}\) with neighbors to form \(\sum_jw_{ij}Y_{j,k}\)
 \State \(X_{i,k+1}=\sum_jw_{ij}Y_{j,k}-\alpha_kS_{i,k}\)
\EndFor
\end{algorithmic}
\end{algorithm}

\noindent\begin{minipage}[t]{0.54\textwidth}
\vspace{0pt}
In stacked form, Algorithm~\ref{alg:plods} is
\begin{equation}
 \Xb_{k+1}=\Wb\Phi(\Xb_k)-\alpha_k\Sb_k.
 \label{eq:plods-update}
\end{equation}
One iteration
uses one communication round, one mini-batch gradient, and the products
\(X_i^\top X_i\) and \(X_i(X_i^\top X_i-I_r)\).
The landing map costs \(O(pr^2)\) local arithmetic and consists entirely of
matrix products and additions.
Figure~\ref{fig:mechanism} illustrates the mechanism.
Within the invariant tube, mixing promotes consensus, while landing with
fixed \(\gamma=1/3\) reduces feasibility error. With a diminishing
stepsize, the iterates approach the Stiefel manifold and achieve vanishing
consensus and average stationarity errors under our assumptions.
\end{minipage}\hfill
\begin{minipage}[t]{0.44\textwidth}
\vspace{0pt}
\includegraphics[width=\linewidth,trim=15 15 15 15,clip]{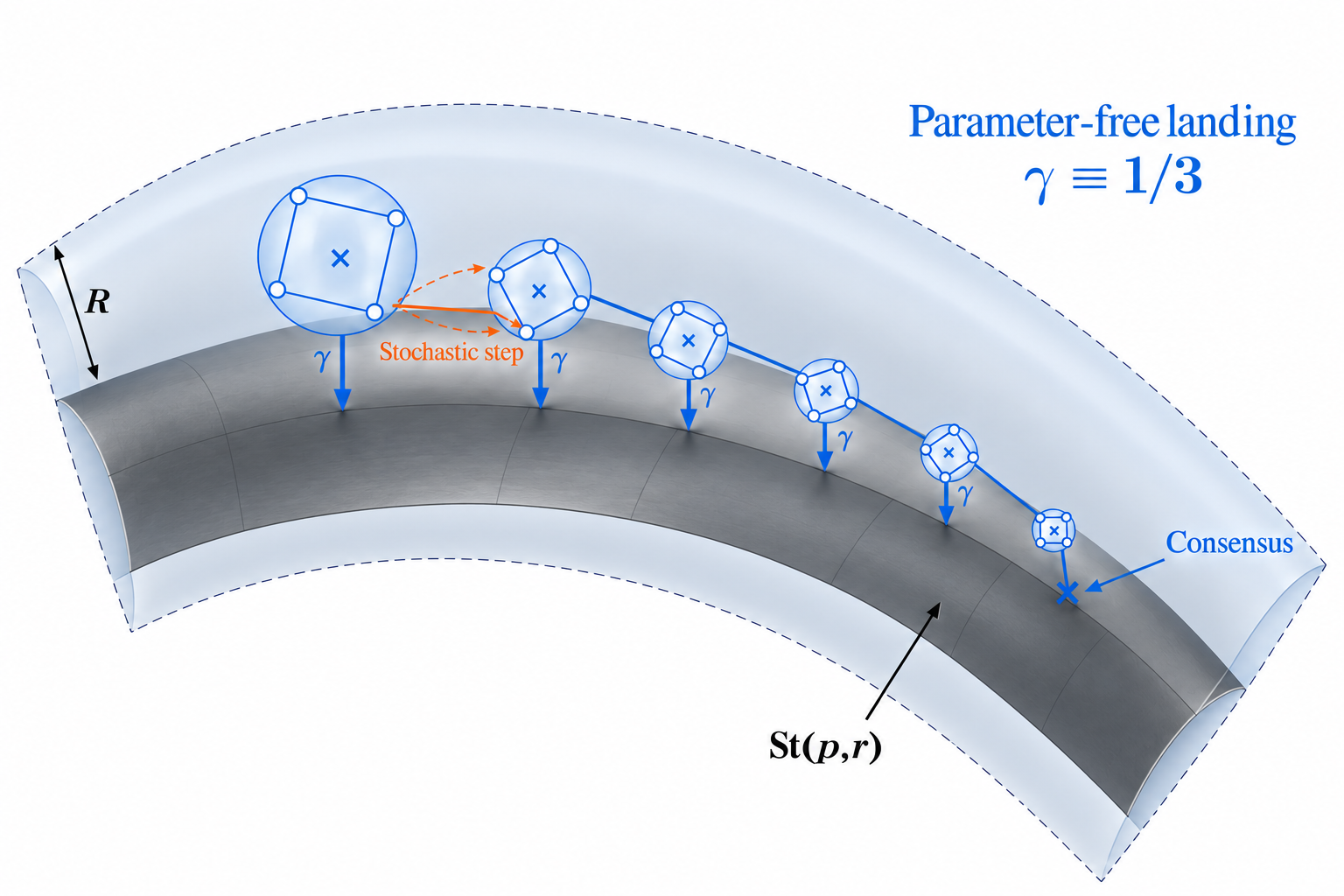}
\setlength{\abovecaptionskip}{2pt}
\captionof{figure}{\small Stochastic landing with \(\gamma=1/3\). Crosses mark network averages.}
\label{fig:mechanism}
\end{minipage}

\section{Convergence Analysis}
\label{sec:theory}

We establish local linear convergence of the penalty-only scheme, then
control stochastic disagreement and objective descent. Together, these
results yield convergence with diminishing stepsizes and linear speedup
with respect to the number of nodes under a horizon-dependent stepsize.
Full proofs and a sufficient stepsize cap appear in
Appendices~\ref{app:geometry}--\ref{app:rates}.

\subsection{Penalty-only scheme and its linear convergence}

The penalty-only scheme \(F(\Xb)=\Wb\Phi(\Xb)\) combines the polynomial
landing map with neighbor communication. We first study how its iteration
reduces the distance to the consensus Stiefel manifold.

Let \(\Psi(\Xb)=\Xb-F(\Xb)\), and let \(\bar\Xb\) be the nearest point to
\(\Xb\) in \(\Cman\).  Set
\(
 \rho_\star:=\max\{\sigma_2(W),1/3\}
\).
The following lemma gives the local geometric property used in the
stochastic analysis.

\begin{lemma}[Restricted secant inequality]
\label{lem:rsi}
For every \(\rho\in(\rho_\star,1)\), there is a radius \(R>0\) such that,
whenever \(d(\Xb)\le R\), the nearest point \(\bar\Xb\in\Cman\) is unique and
\begin{equation}
 \ip{\Psi(\Xb)}{\Xb-\bar\Xb}
 \ge \frac12\norm{\Psi(\Xb)}_F^2
 +\frac{1-\rho^2}{2}d(\Xb)^2.
 \label{eq:rsi}
\end{equation}
\end{lemma}

Although \(\Cman\) intersects a convex consensus subspace with a nonconvex
product Stiefel manifold, the RSI gives a convex-like property near the
feasible set: the residual is aligned with the direction from the nearest
feasible point. Expanding \(\|F(\Xb)-\bar\Xb\|_F^2\) then yields the
following linear convergence result with the prescribed \(\gamma=1/3\).

\begin{theorem}[Linear convergence of the penalty-only scheme]
\label{thm:rsi}
\label{thm:penalty}
For every \(\rho\in(\rho_\star,1)\), after possibly decreasing the radius
\(R\) in Lemma~\ref{lem:rsi}, whenever \(d(\Xb)\le R\),
\begin{equation}
 \norm{F(\Xb)-\bar\Xb}_F
 \le\rho\norm{\Xb-\bar\Xb}_F,
 \qquad
 d(F(\Xb))\le\rho d(\Xb).
 \label{eq:penalty-contraction}
\end{equation}
Consequently, the fixed-point iteration
\(\Xb^{t+1}=F(\Xb^t)\), initialized with \(d(\Xb^0)\le R\), obeys
\begin{equation}
 d(\Xb^t)\le\rho^t d(\Xb^0),
 \qquad t=0,1,\ldots.
 \label{eq:penalty-q-linear}
\end{equation}
\end{theorem}

The local factor \(\rho\) can be any number above
\(\rho_\star=\max\{\sigma_2(W),1/3\}\) on a sufficiently small tube.
Contraction along the normal bundle gives Q-linear decay of the distance
to \(\Cman\), controlling consensus and orthogonality errors together.

\subsection{One-step descent}

We relate feasibility control to objective descent by separating tangent
motion from normal perturbations, then bound the effects of landing,
sampling noise, and disagreement.

\begin{lemma}[Polar projection expansion]
\label{lem:polar-main}
Let \(\bar X=P_{\St}(\widehat X)\), and suppose the singular values of
\(\widehat X\) lie in \([s,S]\subset(0,\infty)\).  For all sufficiently
small \(\Delta\),
\begin{equation}
 P_{\St}(\widehat X+\Delta)-\bar X
 =\mathcal L_{\widehat X}[P_{\bar X}\Delta]
  +\mathcal R_{\widehat X}(\Delta),
 \label{eq:polar-expansion}
\end{equation}
where \(\mathcal L_{\widehat X}=DP_{\St}(\widehat X)\) on
\(T_{\bar X}\St\) is uniformly invertible and
\begin{equation}
 \norm{\mathcal R_{\widehat X}(\Delta)}_F
 \le C_p\norm{P_{\bar X}\Delta}_F\norm{\Delta}_F.
 \label{eq:polar-remainder}
\end{equation}
\end{lemma}

The linear map $L_{\widehat X}$ is characterized by a Sylvester equation
\citep{Higham1986Polar,Gawlik2017PolarDerivative};
Appendix~\ref{app:polar} proves \eqref{eq:polar-remainder} and the
uniform operator bounds. Combining Theorem~\ref{thm:penalty} with the tangent--normal splitting gives
the stochastic disagreement recursion.

\begin{lemma}[Stochastic disagreement recursion]
\label{lem:disagreement}
There are \(C_1>0\) and a cap \(\bar\alpha>0\) such that,
whenever \(\Xb_k\) is in the radius-\(R\) tube and
\(\alpha_k\le\bar\alpha\),
\begin{align}
 \E[d(\Xb_{k+1})^2\mid\mathcal F_k]
 \le\frac{3+\rho^2}{4}d(\Xb_k)^2
 +C_1n\alpha_k^2\bigl(\xi^2+\tfrac{\sigma^2}{b}\bigr)
 +C_1n\alpha_k^4\bigl(\norm{\grad f(\bar X_k)}_F^2+\tfrac{\sigma^2}{nb}\bigr).
 \label{eq:disagreement}
\end{align}
Moreover, for either nonincreasing schedule stated below, the iterates stay
in the tube when \(d(\Xb_0)\le R/2\) and
\(\alpha_0\le(1-\rho)R/(2\sqrt nG)\).
\end{lemma}

The \(\alpha_k^2\) term in \eqref{eq:disagreement} is the usual effect
of heterogeneity and independent local noise.  A gradient shared by all
nodes is instead a common tangent direction; its distance to \(\Cman\)
starts at order \(\alpha_k^2\), producing the \(\alpha_k^4\) term after
squaring.

Double stochasticity of the communication matrix gives the exact average update
\begin{equation}
 \widehat X_{k+1}=\widehat X_k+\widehat F_k-\alpha_k\widehat S_k,
 \qquad \widehat F_k=\frac1n\sum_i\phi(X_{i,k})-\widehat X_k.
 \label{eq:average-update}
\end{equation}
Conditional independence yields
\(\E[\norm{n^{-1}\sum_i(S_{i,k}-G_{i,k})}_F^2\mid\mathcal F_k]\le\sigma^2/(nb)\), while
smoothness, with a tube-uniform constant \(L_g\), gives
\begin{equation}
 \norm{\widehat G_k-\grad f(\bar X_k)}_F^2
 \le\frac{L_g^2}{n}d(\Xb_k)^2.
 \label{eq:gradient-average}
\end{equation}
At $\bar{X}_k$, the local gradients average to the global
gradient; hence the error in \eqref{eq:gradient-average} is controlled by
the disagreement between local iterates.

\begin{lemma}[Landing-average bounds]
\label{lem:landing-average}
There are tube-uniform constants \(L_a,C_a>0\) such that
\begin{equation}
 \norm{\widehat F_k}_F^2\le\frac{L_a^2}{n}d(\Xb_k)^2,
 \qquad
 \norm{P_{\bar X_k}\widehat F_k}_F^2
 \le\frac{C_a}{n^2}d(\Xb_k)^4.
 \label{eq:landing-average-bounds}
\end{equation}
\end{lemma}

The first bound controls the ambient Taylor remainder; the sharper second
bound states that landing perturbs the average only to second order in
tangent directions.  Appendix~\ref{app:convergence} records their direct
derivation from the explicit polynomial landing map.
Applying Lemma~\ref{lem:polar-main} to \eqref{eq:average-update} gives
the following descent estimate.

\begin{lemma}[One-step descent]
\label{lem:descent}
After possibly reducing \(\bar\alpha\), there are constants
\(C_2,C_3>0\) such that every iterate in the invariant tube satisfies
\begin{align}
 \E[f(\bar X_{k+1})-f(\bar X_k)]
 &\le-C_3\alpha_k\E\norm{\grad f(\bar X_k)}_F^2
 +\frac{C_2\alpha_k}{n}\E d(\Xb_k)^2
 +\frac{C_2\alpha_k^2\sigma^2}{nb} \notag\\
 &\quad+\frac{C_2}{n}\E d(\Xb_k)^4
 +\frac{C_2}{\sqrt n}
 \sqrt{\E d(\Xb_k)^4\,\E\norm{\grad f(\bar X_k)}_F^2}.
 \label{eq:descent}
\end{align}
\end{lemma}

The factor \(\alpha_k\) before \(\E d(\Xb_k)^2\) comes from the tangent
projection in \eqref{eq:polar-expansion}. The fourth-order and
mixed terms are reduced in Appendix~\ref{app:accumulation} using tube
invariance and \eqref{eq:disagreement}.

\subsection{Error accumulation and rates}

Summing the disagreement recursion controls its contribution to descent.
The higher-order terms can then be absorbed under a common stepsize cap;
Appendix~\ref{app:accumulation} gives the accumulation argument. Define
\(\mathcal E_0=f(\bar X_0)-\inf_{X\in\St}f(X)
+d(\Xb_0)^2+d(\Xb_0)/\sqrt n\).

\begin{theorem}[Unified stochastic bound]
\label{thm:master}
Under Assumptions~\ref{ass:regularity}--\ref{ass:network}, suppose
the deterministic initialization satisfies \(d(\Xb_0)\le R/2\), and let
\(0<\alpha_k\le\bar\alpha\) be either constant or
\(\alpha_0/\sqrt{k+1}\).  Constants \(C_0,C_1,C_2\), independent of
\(K,b,\sigma,\xi\), satisfy
\begingroup\small
\begin{equation}
 \frac{\sum_{k<K}\alpha_k\E\norm{\grad f(\bar X_k)}_F^2}
      {\sum_{k<K}\alpha_k}
 \le \frac{C_0\mathcal E_0}{\sum_{k<K}\alpha_k}
 +\frac{C_1\sigma^2\sum_{k<K}\alpha_k^2}
 {nb\sum_{k<K}\alpha_k}
 +\frac{C_2n(\xi^2+\sigma^2/b)\sum_{k<K}\alpha_k^3}
 {\sum_{k<K}\alpha_k}.
 \label{eq:master}
\end{equation}
\endgroup
\end{theorem}

The terms in \eqref{eq:master} represent optimization progress,
averaged sampling noise, and accumulated network disagreement, respectively.

\begin{corollary}[Square-root diminishing stepsize]
\label{cor:diminishing}
For \(\alpha_k=\alpha_0/\sqrt{k+1}\),
\begin{equation}
 \min_{k<K}\E\norm{\grad f(\bar X_k)}_F^2
 \le \frac{C_0\mathcal E_0}{\alpha_0\sqrt K}
 +\frac{C_1\alpha_0\sigma^2(1+\log K)}{nb\sqrt K}
 +\frac{C_2n\alpha_0^2(\xi^2+\sigma^2/b)}{\sqrt K}.
 \label{eq:diminishing}
\end{equation}
\end{corollary}

For a horizon-dependent constant step, fix a problem-scale constant \(c>0\),
independent of \(K,n,b,\sigma,\xi\), and set
\begin{equation}
 \alpha_K=\min\!\left\{\bar\alpha,
 \frac{c}{1+\sigma\sqrt{K/(nb)}+(n\xi^2K)^{1/3}}\right\}.
 \label{eq:horizon-step}
\end{equation}

\begin{corollary}[Linear speedup]
\label{cor:speedup}
Use \(\alpha_k\equiv\alpha_K\) from \eqref{eq:horizon-step} for \(K\)
iterations.  For a fixed network, or a uniformly gapped network family with
uniformly bounded problem and oracle constants, every \(\sigma,\xi\ge0\),
including \(\sigma=\xi=0\), satisfies
\begin{equation}
 \frac1K\sum_{k<K}\E\norm{\grad f(\bar X_k)}_F^2
 =O\!\left(\frac{\sigma}{\sqrt{nbK}}
 +\frac{(n\xi^2)^{1/3}}{K^{2/3}}
 +\frac{n^2}{K}\right).
 \label{eq:speed-rate}
\end{equation}
\end{corollary}

The stochastic term \(\sigma/\sqrt{nbK}\) gives linear speedup with
respect to the number of nodes when it dominates the other two terms.
The term \(n^2/K\) is a network-dependent transient.
If \(\sigma>0\), it is dominated by the variance term once
\(K\gtrsim n^5b/\sigma^2\); if \(\xi>0\), it is dominated by the
heterogeneity term once \(K\gtrsim n^5/\xi^2\).
After either applicable threshold, the remaining two terms recover the
classical \(K^{-1/2}\) and \(K^{-2/3}\) iteration exponents of tracker-free
decentralized SGD \citep{Tang2018D2,Lian2017DPSGD}.

\begingroup
\setlength{\textfloatsep}{12pt plus 2pt minus 2pt}
\setlength{\floatsep}{8pt plus 2pt minus 2pt}
\setlength{\abovecaptionskip}{6pt}
\section{Experiments}
\label{sec:experiments}

We evaluate \plados on principal component analysis (PCA) and low-rank matrix completion, examining
stepsize and network sensitivity, scalability, data heterogeneity, and
computational efficiency against DRSGD \citep{Chen2021DecentralizedRG}.
All runs use \(\gamma=1/3\) and 1 communication round with Metropolis
weights \citep{Xiao2006Metropolis}. Curves are medians over 5 seeds for
PCA robustness and scalability, and 3 seeds for matrix completion and
heterogeneity. Appendix~\ref{app:experiments} provides implementation details.

\subsection{Decentralized PCA}

Node \(i\) stores \(A_i\in\R^{m_i\times p}\) and minimizes
\(f_i(X)=-(2m_i)^{-1}\|A_iX\|_F^2\) over \(X\in\St(p,r)\), with \(r=5\).
We evaluate the projected average
\(\bar X_k=P_{\St}(n^{-1}\sum_iX_{i,k})\) using
\(\|\grad f(\bar X_k)\|_F\) and the subspace distance
\begin{equation}
 d_s(\bar X_k,X^\star)=
 \min_{QQ^\top = Q^\top Q=I_r}\|\bar X_k-X^\star Q\|_F.
 \label{eq:experimental-distance}
\end{equation}
Here \(X^\star\) spans the leading \(r\)-dimensional principal subspace.
Figure~\ref{fig:robustness} uses the population covariance as the reference;
the remaining PCA experiments use the global empirical covariance.

\paragraph{Synthetic PCA.}
For stepsize and network sensitivity, we set \(n=8\), \(m_i=1000\) for all
\(i\), \(p=10\), and \(r=5\). The rows of each \(A_i\) are independently
sampled from a zero-mean Gaussian distribution whose covariance has
eigenvalues \(1,0.8,\ldots,0.8^{p-1}\). With batch size 100, each run performs
60000 updates, or 6000 local epochs.
Figure~\ref{fig:robustness}(a,b) compares \plados and DRSGD using
\(\alpha_k=\alpha_0/\sqrt{k+1}\), with
\(\alpha_0\in\{0.01,0.03,0.05,0.08\}\), on an Erd\H{o}s--R\'enyi
graph \citep{Erdos1959Random,Gilbert1959Random} with independent edge
probability \(0.6\), denoted ER(0.6).
Panels (c,d) fix \(\alpha_0=0.03\) and compare complete, star, ER(0.4),
ER(0.6), and ER(0.8) graphs. In panels (a,b), larger values of \(\alpha_0\)
give faster initial progress, and \plados reaches comparable errors in
less wall-clock time than DRSGD. In panels (c,d), the trajectories of both metrics nearly overlap
across the 5 graphs at \(\alpha_0=0.03\).

For scalability, we set \(p=100\) and \(m_i=60000/n\), and compare
\(n=8,16,32\) nodes under a common epoch budget
\citep{Chen2021DecentralizedRG}. With a ring graph and batch size \(b=1\),
each run completes 300 epochs, with the same total sample budget across
node counts. We use the constant stepsize
\begin{equation}
 \alpha_{\rm syn}(n)=\widehat\alpha_n\sqrt{n/300},
 \label{eq:synthetic-scalability-step}
\end{equation}
with \((\widehat\alpha_8,\widehat\alpha_{16},\widehat\alpha_{32})
=(0.003,0.004,0.006)\).
Figure~\ref{fig:scalability}(a,b) shows nearly linear speedup in wall-clock
time over the tested range of node counts.

\begin{figure}[t]
 \centering
 \includegraphics[width=\linewidth]{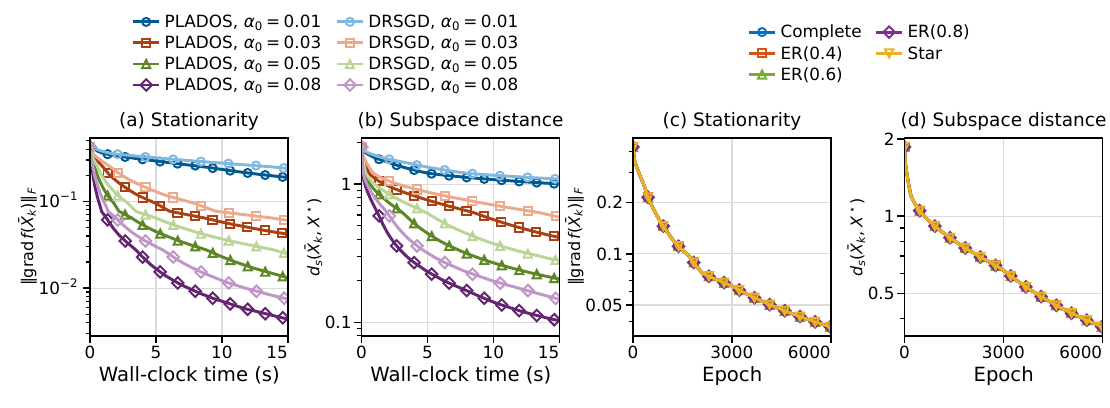}
 \caption{Synthetic PCA: \plados and DRSGD at matched diminishing stepsizes
 (a,b), and \plados on different networks (c,d).}
 \label{fig:robustness}
\end{figure}

\paragraph{MNIST PCA.}
We use the MNIST handwritten-digit dataset \citep{LeCun1998Gradient}.
The scalability experiment uses all 60000 training images. Each image is
vectorized in \(\R^p\) with \(p=784\), and pixel values are divided by 255.
The node counts, local sample sizes, ring graph, batch size, and epoch budget
match the synthetic scalability experiment. We use the constant stepsize
\begin{equation}
 \alpha_{\rm MNIST}(n)=\frac{\widehat\alpha_n\sqrt n}{10000\sqrt{300}},
 \label{eq:mnist-scalability-step}
\end{equation}
with \((\widehat\alpha_8,\widehat\alpha_{16},\widehat\alpha_{32})
=(1.0,1.4,2.0)\).
Figure~\ref{fig:scalability}(c,d) shows aligned epoch curves and nearly linear
wall-clock scaling over the tested range.
Appendix~\ref{app:scalability} details the speedup analysis.

\begin{figure}[t]
 \centering
 \includegraphics[width=\linewidth]{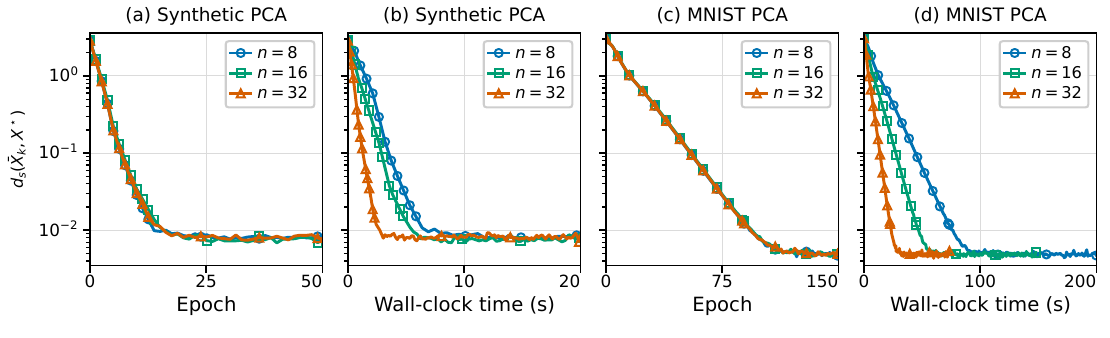}
 \caption{PCA scalability with 8, 16, and 32 MPI nodes.
 All runs complete 300 epochs.}
 \label{fig:scalability}
\end{figure}

To study data heterogeneity, we adapt the label-based partitioning of
\citet{Hsieh2020NonIID}. Images are sorted by label into 8 equal blocks.
Each node retains a randomly selected fraction
\(\rho\in\{0,0.2,0.4,0.6,0.8,1\}\) of its block; the rest is pooled and
randomly redistributed in equal shares. A larger \(\rho\) induces greater
data heterogeneity across nodes through increased label skew. The endpoints
\(\rho=0\) and \(\rho=1\) correspond to random assignment and label-based
blocks, respectively, while the global objective remains unchanged.

The 8 ring nodes use exact local gradients, \(\sigma=0\), and a shared
constant stepsize \(\alpha=7\times10^{-4}\).
Figure~\ref{fig:heterogeneity} shows the Riemannian gradient norm and
subspace distance over epochs 4000--10000. Larger data heterogeneity
generally yields larger Riemannian gradient norms and subspace distances.
Over the full budget of 15000 updates, runs with larger data heterogeneity exhibit nonzero
residual plateaus, illustrating convergence
to a neighborhood at the fixed stepsize, as discussed in
Appendix~\ref{app:heterogeneity}.
This agrees with the residual term in appendix
\eqref{eq:constant-step}.

\begin{figure}[t]
 \centering
 \includegraphics[width=0.780\linewidth]{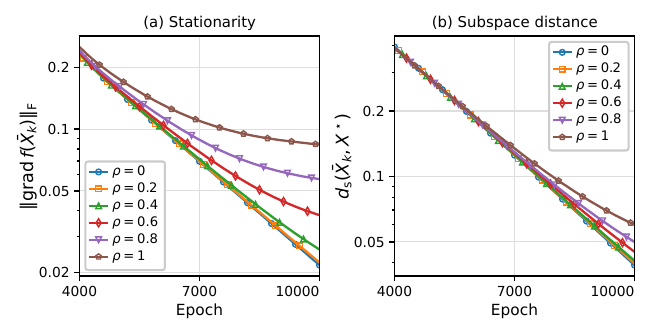}
 \caption{MNIST PCA with heterogeneous partitions.}
 \label{fig:heterogeneity}
\end{figure}

\subsection{Decentralized Low-Rank Matrix Completion}
\label{subsec:lrmc}

We recover a low-rank matrix from partial, noisy observations.
Let \(Y=[y_1,\ldots,y_m]\in\R^{p\times m}\) collect the noisy columns.
The columns are partitioned among \(n\) nodes, with \(\mathcal J_i\)
denoting the indices assigned to node \(i\) and \(m_i=|\mathcal J_i|\).
Node \(i\) minimizes, over \(X\in\St(p,r)\),
\begin{equation}
 f_i(X)=\frac{1}{m_i}\sum_{j\in\mathcal J_i}
 \min_{v\in\R^r}\frac12\|X_{\Omega_j}v-y_{j,\Omega_j}\|_2^2,
 \label{eq:lrmc-objective}
\end{equation}
where \(\Omega_j\subseteq\{1,\ldots,p\}\) contains the observed row indices
of column \(y_j\), and \(X_{\Omega_j}\) consists of the corresponding rows of \(X\).

\enlargethispage{-2\baselineskip}
We set \(p=100\), \(m=12000\), \(r=5\), and compare \(n=8,16,32\) nodes.
The data are generated as \(Y=UV+E\), where \(U\in\St(p,r)\) is the polar
factor of a Gaussian \(p\times r\) matrix, \(V\in\R^{r\times m}\) has
independent standard Gaussian entries, and \(E\in\R^{p\times m}\) has
independent zero-mean Gaussian entries with standard deviation \(10^{-3}\).
Each entry is observed independently with probability \(0.2\).
Columns are randomly divided equally among the nodes, so \(m_i=m/n\).
With a ring graph and batch size \(b=12\), each run completes 800 epochs,
with the same total sample budget across node counts. We use the constant stepsize
\begin{equation}
 \alpha_{\rm LRMC}(n)=\widehat\alpha_n\sqrt{n/800},
 \label{eq:lrmc-scalability-step}
\end{equation}
with \((\widehat\alpha_8,\widehat\alpha_{16},\widehat\alpha_{32})=(0.2,0.3,0.4)\).
Figure~\ref{fig:lrmc} shows aligned epoch curves and nearly linear speedup
in wall-clock time over the tested range of node counts.
The displayed ranges are 150 epochs and 8 seconds; all runs complete 800 epochs.
Additional experiments in Appendix~\ref{app:experiments} show stable
consensus behavior (Figure~\ref{fig:lrmc-consensus}) and further demonstrate
linear speedup (Figure~\ref{fig:paired-speedup}).

\begin{figure}[t]
 \centering
 \includegraphics[width=0.750\linewidth]{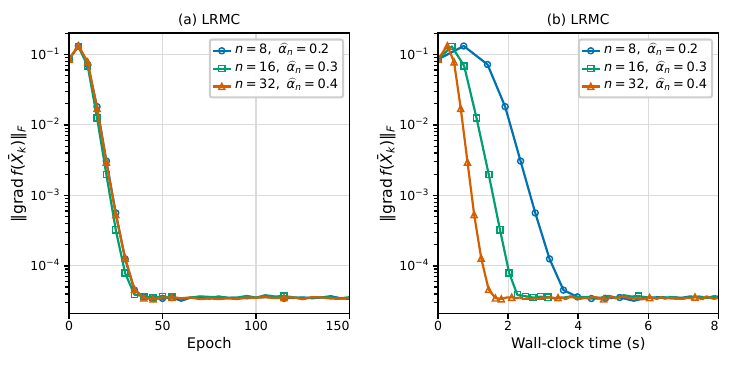}
 \caption{LRMC scalability with 8, 16, and 32 MPI nodes.}
 \label{fig:lrmc}
\end{figure}

\section{Conclusion}

\plados makes the plain decentralized stochastic-gradient recursion
retraction free on the Stiefel manifold while preserving its leading
stochastic and heterogeneity scales.  Its parameter-free landing choice
\(\gamma=1/3\) yields an RSI and a local contraction of distance to the
consensus Stiefel manifold along normal directions; a tangent-preserving polar
expansion then transfers this feasibility control to stochastic descent.
The resulting guarantee accommodates independent local mini-batches and
exhibits linear speedup with respect to the number of nodes after the stated network- and
heterogeneity-dependent transient, without a gradient tracker or repeated
orthonormalization.
Experiments demonstrate nearly linear wall-clock speedup and show that larger
data heterogeneity generally yields larger stationarity residuals at a fixed
stepsize.

\clearpage
\section*{AI Use Statement}
Generative AI tools were used to aid and polish the writing, including
language editing and checks of mathematical exposition. The authors take
full responsibility for the accuracy, integrity, and originality of the
paper.

\section*{Reproducibility Statement}
Appendices~\ref{app:geometry}--\ref{app:rates} provide the proofs, and
Appendix~\ref{app:experiments} describes the experimental implementation,
stepsize scaling, evaluation procedures, and heterogeneous data construction.
The code package is available from the authors upon reasonable request.

\endgroup
\bibliography{references}
\bibliographystyle{iclr2027_conference}

\clearpage
\appendix
\addtocontents{toc}{\protect\setcounter{tocdepth}{2}}
\begingroup
\renewcommand{\contentsname}{Appendix contents}
\tableofcontents
\endgroup
\section{Geometric estimates}
\label{app:geometry}

All norms in this appendix are Frobenius norms unless an operator norm is
displayed explicitly.  We first prove the polar expansion used for the
projected network average and then prove the parameter-free landing
contraction.  The
coefficient is kept as a general \(\gamma\in(0,1)\) in the latter argument and
is set to \(1/3\) only after the dependence on \(\gamma\) is visible.

\subsection{Polar projection with a product remainder}
\label{app:polar}

\begin{lemma}
\label{lem:young-app}
For matrices \(A,B\) of the same size and every \(\eta>0\),
\begin{align}
 2\ip AB&\le \eta\norm A_F^2+\eta^{-1}\norm B_F^2,
 \label{eq:young-inner-app}\\
 \norm{A+B}_F^2
 &\le(1+\eta)\norm A_F^2+(1+\eta^{-1})\norm B_F^2.
 \label{eq:young-square-app}
\end{align}
For nonnegative \(a,b\),
\begin{equation}
 (a+b)^4\le(1+\eta)^3a^4+(1+\eta^{-1})^3b^4.
 \label{eq:young-fourth-app}
\end{equation}
\end{lemma}

\begin{proof}
For \eqref{eq:young-fourth-app}, apply Jensen's inequality
\citep[Section~3.1.8]{Boyd2004Convex}
to \(t\mapsto t^4\) with weights
\(\lambda=(1+\eta)^{-1}\) and
\(1-\lambda=\eta/(1+\eta)\):
\[
 (a+b)^4
 \le \lambda^{-3}a^4+(1-\lambda)^{-3}b^4,
\]
which is \eqref{eq:young-fourth-app}.
\end{proof}

The polar derivative is standard \citep{Gawlik2017PolarDerivative}.
We use its normal-fiber structure to obtain the product remainder below.

\begin{lemma}
\label{lem:polar-app}
Let \(X=QH\) be the right polar decomposition, where
\(Q=P_{\St}(X)\), \(H=(X^\top X)^{1/2}=H^\top\succ0\), and
\(\sigma(H)\subset[s_-,s_+]\) with \(s_->0\).  If
\(\norm\Delta_F\le s_-/2\), then
\begin{align}
 P_{\St}(X+\Delta)-Q
 &=\mathcal L_X[P_Q\Delta]+\mathcal R_X(\Delta),
 \label{eq:polar-expansion-app}\\
 \norm{\mathcal L_X}_{\rm op}&\le s_-^{-1},
 \label{eq:polar-upper-app}\\
 \ip Z{\mathcal L_X[Z]}&\ge s_+^{-1}\norm Z_F^2,
 \qquad Z\in T_Q\St,
 \label{eq:polar-lower-app}\\
 \norm{\mathcal R_X(\Delta)}_F
 &\le C_p\norm{P_Q\Delta}_F\norm\Delta_F.
 \label{eq:polar-remainder-app}
\end{align}
Here
\(P_Q\Delta=\Delta-Q\sym(Q^\top\Delta)\),
\(\mathcal L_X=DP_{\St}(X)|_{T_Q\St}\), and
\(C_p\) depends only on \(s_-\) and \(s_+\).
\end{lemma}

\begin{proof}
The identities \(Q^\top Q=I_r\) and \(X=QH\) give
\begin{equation}
 X^\top X=H Q^\top QH=H^2,
 \qquad
 \sigma_j(X)=\sqrt{\lambda_j(X^\top X)}
 =\lambda_j(H)=\sigma_j(H),
 \quad j=1,\ldots,r.
 \label{eq:polar-singular-values-app}
\end{equation}
The last equality uses \(H\succ0\) \citep[Lemma~2.1]{Higham1986Polar}.
Thus the spectral bounds imposed on
\(H\) are exactly the bounds on the singular values of \(X\).

Only in this proof, put
\[
 \Delta_T:=P_Q\Delta,
 \qquad
 \Delta_N:=Q\sym(Q^\top\Delta).
\]
Then \(\Delta_T\in T_Q\St\), \(\Delta_N\in N_Q\St\), and
\(\Delta=\Delta_T+\Delta_N\).  These two components are orthogonal.  Indeed,
if \(A=\sym(Q^\top\Delta)\), then
\[
 \ip{QA}{\Delta_T}
 =\operatorname{tr}(A Q^\top\Delta_T)=0,
\]
because \(A\) is symmetric whereas \(Q^\top\Delta_T\) is skew symmetric.

The normal segment has constant polar factor.  For \(0\le t\le1\),
\[
 X+t\Delta_N=Q(H+tA),
 \qquad
 \lambda_{\min}(H+tA)
 \ge s_- -\norm\Delta_F\ge s_-/2.
\]
Thus \(H+tA\succ0\), and uniqueness of the right polar decomposition gives
\begin{equation}
 P_{\St}(X+t\Delta_N)=Q,
 \qquad 0\le t\le1.
 \label{eq:normal-fiber-app}
\end{equation}

Every point on the segments below has smallest singular value at least
\(s_-/2\), by Weyl's inequality, and spectral norm at most
\(s_++s_-/2\).  The polar factor is smooth on the full-column-rank set, so
its second derivative is uniformly bounded on the resulting compact set.
Choose \(C_p\) to be twice such a bound.  By
\eqref{eq:normal-fiber-app} and the fundamental theorem of calculus,
\begin{align*}
 P_{\St}(X+\Delta)-P_{\St}(X)
 &=\int_0^1
 DP_{\St}(X+\Delta_N+t\Delta_T)[\Delta_T]dt\\
 &=DP_{\St}(X)[\Delta_T]+\mathcal R_X(\Delta),
\end{align*}
where
\[
 \mathcal R_X(\Delta)
 :=\int_0^1\!
 \{DP_{\St}(X+\Delta_N+t\Delta_T)-DP_{\St}(X)\}
 [\Delta_T]dt.
\]
Applying the fundamental theorem once more to the derivative yields
\begin{align*}
 \norm{\mathcal R_X(\Delta)}_F
 &\le \frac{C_p}2
 \int_0^1(\norm{\Delta_N}_F+t\norm{\Delta_T}_F)
 \norm{\Delta_T}_F\,dt\\
 &\le C_p\norm\Delta_F\norm{\Delta_T}_F.
\end{align*}
This proves equations \plaineqref{eq:polar-expansion-app} and
\plaineqref{eq:polar-remainder-app}.  In particular, a purely normal
perturbation has zero remainder and does not change the polar factor.

It remains to derive the formula for \(\mathcal L_X\) and prove its two
bounds.  For \(Z\in T_Q\St\), write the right polar decomposition of
\(X+tZ\) as \(Q(t)H(t)\) for \(t\) near zero, and put
\(\dot Q=Q'(0)\), \(\dot H=H'(0)\), and
\(\Theta=Q^\top\dot Q\).  Differentiating \(Q(t)^\top Q(t)=I_r\) shows
that \(\Theta^\top=-\Theta\), while differentiating
\(X+tZ=Q(t)H(t)\) gives \(Z=\dot QH+Q\dot H\).
Left multiplication by \(Q^\top\), followed by subtraction of the
transpose, eliminates the symmetric matrix \(\dot H\) and gives
\begin{equation}
 H\Theta+\Theta H=Q^\top Z-Z^\top Q=2Q^\top Z.
 \label{eq:polar-sylvester-app}
\end{equation}
Because every eigenvalue sum \(h_i+h_j\) is positive, this equation has a
unique skew-symmetric solution.  Projecting the differentiated factorization
onto the orthogonal complement of \(Q\) gives
\((I-QQ^\top)\dot Q=(I-QQ^\top)ZH^{-1}\).  Consequently,
\begin{equation}
 \mathcal L_X[Z]=Q\Theta+(I-QQ^\top)ZH^{-1},
 \label{eq:polar-derivative-app}
\end{equation}
in agreement with the standard Fr\'echet derivative formula
\citep{Higham1986Polar,Gawlik2017PolarDerivative}.
Diagonalize \(H=V\operatorname{diag}(h_1,\ldots,h_r)V^\top\) and set
\(B=V^\top Q^\top ZV\).  Because \(Q^\top Z\) is skew symmetric,
\[
 (V^\top\Theta V)_{ij}=\frac{2B_{ij}}{h_i+h_j}.
\]
Consequently,
\[
 \norm\Theta_F\le s_-^{-1}\norm{Q^\top Z}_F,
 \qquad
 \norm{(I-QQ^\top)ZH^{-1}}_F
 \le s_-^{-1}\norm{(I-QQ^\top)Z}_F.
\]
The two terms in \eqref{eq:polar-derivative-app} are orthogonal, as are the
corresponding components of \(Z\); squaring and adding proves
\eqref{eq:polar-upper-app}.  The same eigenbasis calculation gives
\[
 \ip{Q^\top Z}{\Theta}
 =\sum_{i,j}\frac{2B_{ij}^2}{h_i+h_j}
 \ge s_+^{-1}\norm{Q^\top Z}_F^2.
\]
Together with \(H^{-1}\succeq s_+^{-1}I\) in the orthogonal component, this
proves \eqref{eq:polar-lower-app}.  Since the derivative of the
Stiefel-valued curve \(P_{\St}(X+tZ)\) lies in \(T_Q\St\), the map
\(\mathcal L_X\) is an endomorphism of this finite-dimensional tangent
space.  Cauchy--Schwarz and \eqref{eq:polar-lower-app} further give
\[
 \norm{\mathcal L_X[Z]}_F\ge s_+^{-1}\norm Z_F.
\]
Thus \(\mathcal L_X\) is injective and hence invertible on
\(T_Q\St\), with \(\norm{\mathcal L_X^{-1}}_{\rm op}\le s_+\), which
justifies the uniform invertibility asserted in Lemma~\ref{lem:polar-main}.
\end{proof}

\subsection{Parameter-free landing contraction and the RSI}

For this subsection, let \(\gamma\in(0,1)\) and define
\[
 \phi_\gamma(X)=X-\gamma X(X^\top X-I_r),
 \qquad
 \Phi_\gamma(\Xb)=(\phi_\gamma(X_1),\ldots,\phi_\gamma(X_n)),
 \qquad
 F_\gamma=\Wb\Phi_\gamma.
\]

\begin{lemma}
\label{lem:normal-split}
At \(\bar\Xb=\one\otimes Q\in\Cman\),
\begin{equation}
 N_{\bar\Xb}\Cman=
 (\one\otimes N_Q\St)\oplus
 (\one^\perp\otimes\R^{p\times r}),
 \label{eq:normal-split-app}
\end{equation}
and the sum is orthogonal, where
\(N_Q\St=\{QA:A=A^\top\}\).
\end{lemma}

\begin{proof}
The embedded Stiefel tangent space is
\(T_Q\St=\{Z:Q^\top Z+Z^\top Q=0\}\), whose Frobenius-orthogonal
complement is \(N_Q\St=\{QA:A=A^\top\}\)
\citep[Section~7.3]{Boumal2023Manifolds}. Hence
\(T_{\bar\Xb}\Cman=\{\one\otimes Z:Z\in T_Q\St\}\).  For a stack
\(H=(H_1,\ldots,H_n)\), let \(\widehat H=n^{-1}\sum_iH_i\).  Then
\[
 \sum_i\ip{H_i}{Z}=n\ip{\widehat H}{Z}
 \quad\text{for every }Z\in T_Q\St.
\]
Thus \(H\in N_{\bar\Xb}\Cman\) if and only if
\(\widehat H\in N_Q\St\).  In that case,
\[
 H=\one\otimes\widehat H+(H-\one\otimes\widehat H),
\]
where the second term has zero network average.  This proves
\eqref{eq:normal-split-app}; orthogonality follows from the zero average.
\end{proof}

\begin{lemma}
\label{lem:normal-derivative}
Let
\(\rho_\star(\gamma)=\max\{\sigma_2(W),|1-2\gamma|\}\).  At every
\(\bar\Xb\in\Cman\),
\begin{align}
 DF_\gamma(\bar\Xb)[u]&=u,
 &&u\in T_{\bar\Xb}\Cman,
 \label{eq:tangent-identity-app}\\
 \norm{DF_\gamma(\bar\Xb)[v]}_F
 &\le\rho_\star(\gamma)\norm v_F,
 &&v\in N_{\bar\Xb}\Cman.
 \label{eq:normal-linearization-app}
\end{align}
Moreover, choose
\(\rho\in(\rho_\star(\gamma),1)\), and suppose
\begin{equation}
 \gamma(6R+3R^2)\le\rho-\rho_\star(\gamma).
 \label{eq:normal-radius-app}
\end{equation}
If \(\bar\Xb\in\Cman\), \(v\in N_{\bar\Xb}\Cman\),
\(\norm v_F=1\), and \(\norm{\mathbf Y-\bar\Xb}_F\le R\), then
\begin{equation}
 \norm{DF_\gamma(\mathbf Y)[v]}_F\le\rho.
 \label{eq:normal-derivative-app}
\end{equation}
\end{lemma}

\begin{proof}
Write \(\bar\Xb=\one\otimes Q\).  At the feasible block \(Q\),
\[
 D\phi_\gamma(Q)[H]
 =H-\gamma Q(H^\top Q+Q^\top H).
\]
Every ambient matrix has the orthogonal decomposition
\(H=P_QH+Q\sym(Q^\top H)\).  The derivative is the identity on the first
component and multiplication by \(1-2\gamma\) on the second.  In particular,
its operator norm is one.

If \(u\in T_{\bar\Xb}\Cman\), then
\(u=\one\otimes H\) for some \(H\in T_Q\St\).  The blockwise landing
derivative fixes \(u\), and \(W\one=\one\), so communication also fixes it. This
proves \eqref{eq:tangent-identity-app}.

By Lemma~\ref{lem:normal-split}, write the orthogonal decomposition
\[
 v=\one\otimes QA+z,
 \qquad A=A^\top,
 \qquad \widehat z=0.
\]
Communication fixes the consensus-normal component, on which the landing derivative
acts by \(1-2\gamma\).  The same blockwise derivative maps the zero-average
subspace into itself: its network average is
\(D\phi_\gamma(Q)[\widehat z]=0\).  Afterward \(W-J\) contracts it by
\(\sigma_2(W)=\norm{W-J}_2\).  The two images remain orthogonal, and hence
\begin{align*}
 \norm{DF_\gamma(\bar\Xb)[v]}_F^2
 &\le |1-2\gamma|^2\norm{\one\otimes QA}_F^2
 +\sigma_2(W)^2\norm z_F^2\\
 &\le\rho_\star(\gamma)^2\norm v_F^2.
\end{align*}
By homogeneity this proves \eqref{eq:normal-linearization-app} for every
normal vector.

It remains to control the derivative away from \(\Cman\).  For one block,
write \(Y=Q+E\).  Direct differentiation gives
\[
 D\phi_\gamma(Y)[H]
 =H-\gamma\{H(Y^\top Y-I_r)+Y(H^\top Y+Y^\top H)\}.
\]
The first difference from the corresponding expression at \(Q\) is bounded
by \((2\norm E_F+\norm E_F^2)\norm H_F\).  Expanding the second difference
into
\[
 E(H^\top Q+Q^\top H)
 +Q(H^\top E+E^\top H)
 +E(H^\top E+E^\top H)
\]
bounds it by \((4\norm E_F+2\norm E_F^2)\norm H_F\).  Therefore
\begin{equation}
 \norm{(D\phi_\gamma(Y)-D\phi_\gamma(Q))[H]}_F
 \le\gamma(6R+3R^2)\norm H_F
 \label{eq:landing-derivative-lipschitz-app}
\end{equation}
whenever \(\norm E_F\le R\).  Applying this blockwise, using
\(\norm{W}_2=1\), and then combining it with the contraction at
\(\bar\Xb\) gives
\[
 \norm{DF_\gamma(\mathbf Y)[v]}_F
 \le\{\rho_\star(\gamma)+\gamma(6R+3R^2)\}\norm v_F
 \le\rho,
\]
which is \eqref{eq:normal-derivative-app}.
\end{proof}

\begin{proof}[Proof of Lemma~\ref{lem:rsi} and Theorem~\ref{thm:rsi}]
By the smooth nearest-point projection theorem
\citep[Theorem~5.53]{Boumal2023Manifolds}, shrink \(R\), if necessary, so
that the radius-\(R\) neighborhood is a tubular neighborhood of \(\Cman\).
Let
\(\bar\Xb=\one\otimes\bar X\) be the unique nearest point to \(\Xb\).
The tubular-neighborhood property gives
\(\Xb-\bar\Xb\in N_{\bar\Xb}\Cman\), and every point of the normal
segment
\[
 \bar\Xb+t(\Xb-\bar\Xb),\qquad 0\le t\le1,
\]
is within distance \(R\) of \(\bar\Xb\).  Moreover,
\(F_\gamma(\bar\Xb)=\bar\Xb\).  The fundamental theorem of calculus and
Lemma~\ref{lem:normal-derivative} therefore give the stronger,
fixed-representative estimate
\begin{align}
 \norm{F_\gamma(\Xb)-\bar\Xb}_F
 &=\left\|\int_0^1DF_\gamma(\bar\Xb+t(\Xb-\bar\Xb))
 [\Xb-\bar\Xb]dt\right\|_F\notag\\
 &\le\rho\norm{\Xb-\bar\Xb}_F.
 \label{eq:fixed-representative-app}
\end{align}
Taking the distance from \(F_\gamma(\Xb)\) to the nearest point in
\(\Cman\) can only decrease the left-hand side, so
\(d(F_\gamma(\Xb))\le\rho d(\Xb)\).

Now set \(\gamma=1/3\), for which
\(\rho_\star(1/3)=\max\{\sigma_2(W),1/3\}\).  For example, after also imposing
the tubular-radius restriction, any
\[
 0<R\le\min\left\{\frac18,
 -1+\sqrt{1+\rho-\max\{\sigma_2(W),1/3\}}\right\}
\]
satisfies \eqref{eq:normal-radius-app}.  This proves the contraction theorem.

Finally, let \(F=F_{1/3}\) and
\(\Psi(\Xb)=\Xb-F(\Xb)\). With
\(e=\Xb-\bar\Xb\) and \(r=\Psi(\Xb)\), the fixed-representative estimate
is \(\norm{e-r}_F^2\le\rho^2\norm e_F^2\).  Expanding its left-hand side
and rearranging gives the restricted secant inequality
\begin{equation}
 \ip{\Psi(\Xb)}{\Xb-\bar\Xb}
 \ge\frac12\norm{\Psi(\Xb)}_F^2
 +\frac{1-\rho^2}{2}d(\Xb)^2.
 \label{eq:rsi-app}
\end{equation}
Thus both the contraction and the RSI use the same fixed nearest
representative. Conversely, substituting the RSI back into the
exact identity
\[
 \norm{F(\Xb)-\bar\Xb}_F^2
 =d(\Xb)^2+\norm{\Psi(\Xb)}_F^2
 -2\ip{\Psi(\Xb)}{\Xb-\bar\Xb}
\]
recovers the fixed-representative contraction, so the two estimates are
equivalent on this tube.

Finally, if \(\Xb^{t+1}=F(\Xb^t)\) and \(d(\Xb^0)\le R\), then
\(d(\Xb^{t+1})\le\rho d(\Xb^t)\le R\).  Induction gives
\(d(\Xb^t)\le\rho^t d(\Xb^0)\). At \(\gamma=1/3\), equations
\plaineqref{eq:tangent-identity-app} and
\plaineqref{eq:normal-linearization-app} also give, at every
\(\bar\Xb\in\Cman\),
\begin{equation}
 DF(\bar\Xb)[u]=u\quad(u\in T_{\bar\Xb}\Cman),
 \qquad
 \norm{DF(\bar\Xb)|_{N_{\bar\Xb}\Cman}}_{\rm op}
 \le \rho_\star<1.
 \label{eq:restricted-linearization}
\end{equation}
This identifies the tangent and normal behavior of the linearization.
\end{proof}

\section{Stochastic recursion and one-step descent}
\label{app:convergence}

Let \(\mathcal F_k\) be the sigma-field generated by the initialization and
all mini-batches (and any oracle randomness) strictly before time \(k\).
Equivalently, it is the history immediately before the time-\(k\)
mini-batches are drawn. Thus \(\Xb_k\),
\(\bar X_k\), and all population gradients evaluated at \(\Xb_k\) are
\(\mathcal F_k\)-measurable.  Compactness and
Assumptions~\ref{ass:regularity}--\ref{ass:network} give finite constants
\(L_g,L_f,G_f\) such that, throughout the tube,
\begin{align}
 \norm{P_X\nabla f_i(X)-P_Y\nabla f_i(Y)}_F
 &\le L_g\norm{X-Y}_F,
 \label{eq:local-direction-lipschitz-app}\\
 \norm{\grad f(Q)}_F&\le G_f,
 \qquad Q\in\St,
 \label{eq:global-gradient-bound-app}\\
 f(Q')-f(Q)
 &\le\ip{\grad f(Q)}{Q'-Q}
 +\frac{L_f}{2}\norm{Q'-Q}_F^2,
 \qquad Q,Q'\in\St.
 \label{eq:chord-smoothness-app}
\end{align}
These bounds follow from the stated local regularity on the compact tube.
Since \(P_X\) is uniformly bounded there, we enlarge \(G\) so that
\(G\ge1\) bounds both the batch gradient and \(\|S_{i,k}\|_F\).
This common bound will be used in the pathwise tube-invariance argument.

\subsection{Noise and projected-gradient estimates}

Define
\[
 G_{i,k}:=P_{X_{i,k}}\nabla f_i(X_{i,k}),
 \qquad
 \Xi_{i,k}:=S_{i,k}-G_{i,k},
\]
and let \(\Gb_k\) and \(\boldsymbol\Xi_k\) denote the corresponding stacks.
The raw-oracle assumption, the \(\mathcal F_k\)-measurability of
\(X_{i,k}\), and linearity of \(U\mapsto P_{X_{i,k}}U\) give
\[
 \E[S_{i,k}\mid\mathcal F_k]
 =P_{X_{i,k}}\E\!\left[\frac1b\sum_{\ell=1}^b
 \nabla_XF_i(X_{i,k};\zeta_{i,k}^{(\ell)})
 \,\middle|\,\mathcal F_k\right]
 =P_{X_{i,k}}\nabla f_i(X_{i,k})=G_{i,k}.
\]
Thus \(\Sb_k=\Gb_k+\boldsymbol\Xi_k\) and
\(\E[\Xi_{i,k}\mid\mathcal F_k]=0\). Because conditionally independent raw mini-batches
remain conditionally independent after the deterministic maps
\(U\mapsto P_{X_{i,k}}U\), we have
\begin{align}
 \E[\Xi_{i,k}\mid\mathcal F_k]&=0,
 \label{eq:noise-zero-app}\\
 \E[\norm{\widehat{\boldsymbol\Xi}_k}_F^2\mid\mathcal F_k]
 &\le\frac{\sigma^2}{nb},
 \label{eq:average-noise-app}\\
 \E[\norm{(I-\Jb)\boldsymbol\Xi_k}_F^2\mid\mathcal F_k]
 &\le\frac{n\sigma^2}{b}.
 \label{eq:centered-noise-app}
\end{align}
Indeed, expanding the first squared norm conditionally gives
\begin{align*}
 \E\!\left[\left\|\frac1n\sum_i\Xi_{i,k}\right\|_F^2
 \,\middle|\,\mathcal F_k\right]
 &=\frac1{n^2}\sum_i\E[\norm{\Xi_{i,k}}_F^2\mid\mathcal F_k]\\
 &\quad+\frac1{n^2}\sum_{i\ne j}
 \E[\ip{\Xi_{i,k}}{\Xi_{j,k}}\mid\mathcal F_k].
\end{align*}
For \(i\ne j\), conditional independence and \eqref{eq:noise-zero-app}
give
\[
 \E[\ip{\Xi_{i,k}}{\Xi_{j,k}}\mid\mathcal F_k]
 =\ip{\E[\Xi_{i,k}\mid\mathcal F_k]}
 {\E[\Xi_{j,k}\mid\mathcal F_k]}=0.
\]
This proves \eqref{eq:average-noise-app}.  Since
\(I-\Jb\) is an orthogonal projector,
\[
 \E[\norm{(I-\Jb)\boldsymbol\Xi_k}_F^2\mid\mathcal F_k]
 \le\sum_i\E[\norm{\Xi_{i,k}}_F^2\mid\mathcal F_k],
\]
which proves \eqref{eq:centered-noise-app}.

At the common feasible point \(\bar X_k\), linearity of
\(P_{\bar X_k}\) and the definition of \(f\) give
\(n^{-1}\sum_i\grad f_i(\bar X_k)=\grad f(\bar X_k)\). Moreover,
nonexpansiveness of the feasible tangent projection and the heterogeneity
assumption imply
\begin{equation}
 \sum_i\norm{\grad f_i(\bar X_k)-\grad f(\bar X_k)}_F^2
 \le\sum_i\norm{\nabla f_i(\bar X_k)-\nabla f(\bar X_k)}_F^2
 \le n\xi^2.
 \label{eq:feasible-heterogeneity-app}
\end{equation}
By \eqref{eq:local-direction-lipschitz-app} and Jensen's inequality,
\begin{align}
 \norm{\widehat G_k-\grad f(\bar X_k)}_F^2
 &=\left\|\frac1n\sum_i(G_{i,k}-\grad f_i(\bar X_k))\right\|_F^2\notag\\
 &\le\frac1n\sum_i\norm{G_{i,k}-\grad f_i(\bar X_k)}_F^2
 \le\frac{L_g^2}{n}d(\Xb_k)^2.
 \label{eq:average-gradient-gap-app}
\end{align}
For the centered stack, insert the local gradients evaluated at
\(\bar X_k\). Using \(\norm{I-\Jb}_2\le1\),
\(\norm{A+B}_F^2\le2\norm A_F^2+2\norm B_F^2\), and
\eqref{eq:feasible-heterogeneity-app} gives
\begin{align*}
 \norm{(I-\Jb)\Gb_k}_F^2
 &\le2\sum_i\norm{G_{i,k}-\grad f_i(\bar X_k)}_F^2
 +2\sum_i\norm{\grad f_i(\bar X_k)-\grad f(\bar X_k)}_F^2\\
 &\le2L_g^2d(\Xb_k)^2+2n\xi^2.
\end{align*}
The population stack is measurable and the noise has zero conditional mean,
so the conditional cross term vanishes.  Together with
\eqref{eq:centered-noise-app}, this yields
\begin{equation}
 \E[\norm{(I-\Jb)\Sb_k}_F^2\mid\mathcal F_k]
 \le2L_g^2d(\Xb_k)^2+2n\xi^2+\frac{n\sigma^2}{b}.
 \label{eq:centered-signal-app}
\end{equation}
The same orthogonality for the averaged noise gives
\begin{align}
 \E[\norm{\widehat S_k}_F^2\mid\mathcal F_k]
 &=\norm{\widehat G_k}_F^2
 +\E[\norm{\widehat{\boldsymbol\Xi}_k}_F^2\mid\mathcal F_k]\notag\\
 &\le2\norm{\grad f(\bar X_k)}_F^2+\frac{2L_g^2}{n}d(\Xb_k)^2
 +\frac{\sigma^2}{nb}.
 \label{eq:average-signal-app}
\end{align}
Finally, the almost-sure oracle bound, enlarged at the start of this appendix,
implies the pathwise estimates
\begin{equation}
 \norm{\Sb_k}_F\le\sqrt nG,
 \qquad
 \norm{\widehat S_k}_F\le G.
 \label{eq:pathwise-signal-app}
\end{equation}

\subsection{The disagreement recursion}

The next estimate controls the normal component of the stochastic direction
at \(\bar X_k\). Using \((I-P_{\bar X_k})P_{\bar X_k}=0\), we obtain
\begin{align*}
 (I-P_{\bar X_k})S_{i,k}
 &=(I-P_{\bar X_k})(P_{X_{i,k}}-P_{\bar X_k})
 \left(\frac1b\sum_{\ell=1}^b
 \nabla_XF_i(X_{i,k};\zeta_{i,k}^{(\ell)})\right).
\end{align*}
From
\[
 (P_X-P_{\bar X_k})U
 =-(X-\bar X_k)\sym(X^\top U)
 -\bar X_k\sym((X-\bar X_k)^\top U)
\]
and \(\norm X_2\le1+R\) in the tube, we obtain
\begin{equation}
 \norm{(I-P_{\bar X_k})S_{i,k}}_F
 \le L_P G\,\norm{X_{i,k}-\bar X_k}_F,
 \qquad L_P:=2+R.
 \label{eq:normal-signal-app}
\end{equation}

\begin{lemma}
\label{lem:tangent-perturb}
Fix an iteration \(k\) and let \(\mathbf Y_k=F(\Xb_k)\). If
\(\alpha_k\norm{\widehat S_k}_F\le1/2\), then, for every \(\eta>0\),
\begin{align}
 d(\mathbf Y_k-\alpha_k\Sb_k)^2
 &\le(1+\eta)\norm{\mathbf Y_k-\one\otimes\bar X_k}_F^2\notag\\
 &\quad+C_t(\eta)\alpha_k^2
 \left\{\norm{(I-\Jb)\Sb_k}_F^2
 +\sum_i\norm{(I-P_{\bar X_k})S_{i,k}}_F^2\right\}\notag\\
 &\quad+C_t(\eta)n\alpha_k^4\norm{\widehat S_k}_F^4,
 \label{eq:common-tangent-app}
\end{align}
where
\(C_t(\eta)=2(1+\eta^{-1})
\max\{1,C_p^2\}\).
\end{lemma}

\begin{proof}
Use the feasible comparison point
\(P_{\St}(\bar X_k-\alpha_kP_{\bar X_k}\widehat S_k)\).
Lemma~\ref{lem:polar-app}, applied at \(\bar X_k\), has
\(\mathcal L_{\bar X_k}[P_{\bar X_k}\widehat S_k]
=P_{\bar X_k}\widehat S_k\). Hence
\begin{equation}
 \begin{aligned}
 &P_{\St}(\bar X_k-\alpha_kP_{\bar X_k}\widehat S_k)\\
 &\quad=\bar X_k-\alpha_kP_{\bar X_k}\widehat S_k
 +\mathcal R_{\bar X_k}(-\alpha_kP_{\bar X_k}\widehat S_k),\\
 &\norm{\mathcal R_{\bar X_k}(-\alpha_kP_{\bar X_k}\widehat S_k)}_F
 \le C_p\alpha_k^2\norm{P_{\bar X_k}\widehat S_k}_F^2.
 \end{aligned}
 \label{eq:comparison-retraction-app}
\end{equation}
The condition \(\alpha_k\norm{\widehat S_k}_F\le1/2\) guarantees the radius
required in that lemma.  Direct subtraction gives
\begin{align*}
 &\mathbf Y_k-\alpha_k\Sb_k
 -\one\otimes P_{\St}(\bar X_k-\alpha_kP_{\bar X_k}\widehat S_k)\\
 &\quad=\mathbf Y_k-\one\otimes\bar X_k
 -\alpha_k(\Sb_k-\one\otimes P_{\bar X_k}\widehat S_k)\\
 &\qquad-\one\otimes\mathcal R_{\bar X_k}
 (-\alpha_kP_{\bar X_k}\widehat S_k).
\end{align*}
The middle bracket has the exact orthogonal decomposition
\begin{equation}
 \Sb_k-\one\otimes P_{\bar X_k}\widehat S_k
 =(I-\Jb)\Sb_k+\one\otimes(I-P_{\bar X_k})\widehat S_k.
 \label{eq:tangent-stack-split-app}
\end{equation}
Indeed, the first term has zero network average and the second is a consensus
stack.  Jensen's inequality further gives
\begin{equation}
 n\norm{(I-P_{\bar X_k})\widehat S_k}_F^2
 \le\sum_i\norm{(I-P_{\bar X_k})S_{i,k}}_F^2.
 \label{eq:normal-average-jensen-app}
\end{equation}
Apply \eqref{eq:young-square-app} first to the initial error and the sum of
the last two terms, and then use \(\norm{B+C}_F^2\le2\norm B_F^2+2\norm C_F^2\).
Equations \plaineqref{eq:comparison-retraction-app}--
\plaineqref{eq:normal-average-jensen-app}, together with
\(\norm{P_{\bar X_k}\widehat S_k}_F\le\norm{\widehat S_k}_F\), give
\eqref{eq:common-tangent-app}.  Finally, distance to \(\Cman\) is no larger
than the norm to the particular feasible stack
\(\one\otimes P_{\St}(\bar X_k-\alpha_kP_{\bar X_k}\widehat S_k)\),
which completes the proof.
\end{proof}

\begin{proof}[Proof of Lemma~\ref{lem:disagreement}]
Put \(\mathbf Y_k=F(\Xb_k)\) and use the current nearest
representative \(\bar\Xb_k=\one\otimes\bar X_k\). The fixed-representative contraction
\eqref{eq:fixed-representative-app} gives
\begin{equation}
 \norm{\mathbf Y_k-\bar\Xb_k}_F
 \le\rho d(\Xb_k).
 \label{eq:fixed-representative-use-app}
\end{equation}
Moreover, \eqref{eq:normal-signal-app} gives
\begin{equation}
 \sum_i\norm{(I-P_{\bar X_k})S_{i,k}}_F^2
 \le L_P^2G^2d(\Xb_k)^2.
 \label{eq:normal-residual-use-app}
\end{equation}
Temporarily impose \(\alpha_kG\le1/2\); the fiber restriction
\eqref{eq:fiber-cap-app} below is strictly stronger.  Apply
Lemma~\ref{lem:tangent-perturb} with
\(\eta=(1-\rho^2)/(2\rho^2)\). Conditional expectation, followed by equations
\plaineqref{eq:centered-signal-app} and
\plaineqref{eq:average-signal-app}, and the
pathwise inequality
\(\norm{\widehat S_k}_F^4\le G^2\norm{\widehat S_k}_F^2\), yields
\begin{align}
 \E[d(\Xb_{k+1})^2\mid\mathcal F_k]
 &\le\frac{1+\rho^2}{2}d(\Xb_k)^2
 +C_t\alpha_k^2
 \left\{(2L_g^2+L_P^2G^2)d(\Xb_k)^2
 +2n\xi^2+\frac{n\sigma^2}{b}\right\}\notag\\
 &\quad+C_tnG^2\alpha_k^4
 \left\{2\norm{\grad f(\bar X_k)}_F^2
 +\frac{2L_g^2}{n}d(\Xb_k)^2
 +\frac{\sigma^2}{nb}\right\},
 \label{eq:disagreement-preabsorb-app}
\end{align}
where \(C_t=2(1+\rho^2)(1-\rho^2)^{-1}\max\{1,C_p^2\}\). Let
\begin{equation}
 C_d:=C_t
 \{2L_g^2+L_P^2G^2+2G^2L_g^2\}.
 \label{eq:disagreement-constant-app}
\end{equation}
Because \(\alpha_k\le1\), all multiples of \(d(\Xb_k)^2\) after
\((1+\rho^2)/2\) in \eqref{eq:disagreement-preabsorb-app} are at most
\(C_d\alpha_k^2d(\Xb_k)^2\).  Therefore the explicit condition
\begin{equation}
 \alpha_k\le\alpha_d:=\sqrt{\frac{1-\rho^2}{4C_d}}
 \label{eq:disagreement-cap-app}
\end{equation}
gives
\begin{equation}
 \frac{1+\rho^2}{2}+C_d\alpha_k^2
 \le\frac{3+\rho^2}{4}<1.
 \label{eq:disagreement-margin-app}
\end{equation}
The explicit choice
\begin{equation}
 C_1:=2C_t\max\{1,G^2\}
 \label{eq:disagreement-forcing-constant-app}
\end{equation}
dominates every remaining coefficient, independently of
\(k,n,b,\sigma,\xi\), and gives
\begin{align}
 \E[d(\Xb_{k+1})^2\mid\mathcal F_k]
 &\le\frac{3+\rho^2}{4}d(\Xb_k)^2
 +C_1n\alpha_k^2\left(\xi^2+\frac{\sigma^2}{b}\right)\notag\\
 &\quad+C_1n\alpha_k^4
 \left(\norm{\grad f(\bar X_k)}_F^2+\frac{\sigma^2}{nb}\right),
 \label{eq:disagreement-final-app}
\end{align}
which is the claimed stochastic recursion.

It remains to justify that every estimate above is applied inside the tube.
Nearest-point minimality and equations
\plaineqref{eq:fixed-representative-app} and
\plaineqref{eq:pathwise-signal-app} give the pathwise inequality
\begin{align}
 d(\Xb_{k+1})
 &\le d(F(\Xb_k))+\alpha_k\norm{\Sb_k}_F\notag\\
 &\le\rho d(\Xb_k)+\alpha_k\sqrt nG.
 \label{eq:pathwise-tube-recursion-app}
\end{align}
If \(d(\Xb_0)\le R/2\) and
\begin{equation}
 \alpha_0\le\alpha_t:=\frac{(1-\rho)R}{2\sqrt nG},
 \label{eq:tube-cap-app}
\end{equation}
then the induction hypothesis \(d(\Xb_k)\le R\) implies
\[
 d(\Xb_{k+1})
 \le\rho R+\frac{1-\rho}{2}R
 =\frac{1+\rho}{2}R<R.
\]
The initialization starts the induction, proving tube invariance.
\end{proof}

\subsection{Landing average and polar radii}

For the algorithmic value \(\gamma=1/3\), write
\[
 \widehat F_k=-\frac{\gamma}{n}
 \sum_iX_{i,k}(X_{i,k}^\top X_{i,k}-I_r).
\]
Since \(\bar X_k^\top\bar X_k=I_r\),
\[
 X_{i,k}^\top X_{i,k}-I_r=(X_{i,k}-\bar X_k)^\top X_{i,k}
 +\bar X_k^\top(X_{i,k}-\bar X_k).
\]
Thus, on the radius-\(R\) tube,
\begin{equation}
 \norm{X_{i,k}^\top X_{i,k}-I_r}_F
 \le(2+R)\norm{X_{i,k}-\bar X_k}_F,
 \qquad
 \norm{X_{i,k}}_2\le1+R.
 \label{eq:constraint-residual-app}
\end{equation}
Set
\begin{equation}
 L_a:=\gamma(1+R)(2+R),
 \qquad
 C_a:=\gamma^2(2+R)^2.
 \label{eq:landing-constants-app}
\end{equation}
Using \(\norm{\sum_iA_i}_F^2\le n\sum_i\norm{A_i}_F^2\) gives
\begin{equation}
 \norm{\widehat F_k}_F^2
 \le\frac{L_a^2}{n}d(\Xb_k)^2.
 \label{eq:full-landing-app}
\end{equation}
For the tangent component, symmetry of \(X_{i,k}^\top X_{i,k}-I_r\) implies
\(P_{\bar X_k}[\bar X_k(X_{i,k}^\top X_{i,k}-I_r)]=0\), and hence
\[
 P_{\bar X_k}[X_{i,k}(X_{i,k}^\top X_{i,k}-I_r)]
 =P_{\bar X_k}[(X_{i,k}-\bar X_k)(X_{i,k}^\top X_{i,k}-I_r)].
\]
The feasible tangent projection is nonexpansive, so
\begin{align*}
 \norm{P_{\bar X_k}\widehat F_k}_F
 &\le\frac{\gamma(2+R)}{n}
 \sum_i\norm{X_{i,k}-\bar X_k}_F^2.
\end{align*}
After squaring,
\begin{equation}
 \norm{P_{\bar X_k}\widehat F_k}_F^2
 \le\frac{C_a}{n^2}d(\Xb_k)^4
 \le\frac{C_a}{n}d(\Xb_k)^4.
 \label{eq:tangent-landing-app}
\end{equation}

Jensen's inequality gives
\(\norm{\widehat X_k-\bar X_k}_F\le d(\Xb_k)/\sqrt n\).  Write the
right polar decomposition as
\(\widehat X_k=\bar X_kH_k\), where
\(H_k=(\widehat X_k^\top\widehat X_k)^{1/2}\succ0\).  By
\eqref{eq:polar-singular-values-app},
\begin{equation}
 \sigma_j(\widehat X_k)=\lambda_j(H_k)=\sigma_j(H_k),
 \qquad j=1,\ldots,r.
 \label{eq:average-polar-singular-values-app}
\end{equation}
Moreover,
\(\norm{H_k-I_r}_F=\norm{\widehat X_k-\bar X_k}_F\), because
\(\bar X_k^\top\bar X_k=I_r\).  Hence every eigenvalue of \(H_k\) lies in
\([r_-,r_+]\), where
\begin{equation}
 r_-:=1-\frac{R}{\sqrt n},
 \qquad
 r_+:=1+\frac{R}{\sqrt n}.
 \label{eq:polar-radii-app}
\end{equation}
In particular, \(R\le1/8\) gives \(r_-\ge7/8\).  The averaged update is
\begin{equation}
 \widehat X_{k+1}-\widehat X_k
 =\Delta_k:=\widehat F_k-\alpha_k\widehat S_k.
 \label{eq:average-update-app}
\end{equation}
Equations \plaineqref{eq:pathwise-signal-app} and
\plaineqref{eq:full-landing-app}
imply
\[
 \norm{\Delta_k}_F
 \le\frac{L_aR}{\sqrt n}+\alpha_kG.
\]
The numerator in the following quantity is positive:
\begin{equation}
 \alpha_f:=\frac{r_-/2-L_aR/\sqrt n}{G}>0.
 \label{eq:fiber-cap-app}
\end{equation}
Indeed, \(r_-/2\ge7/16\), whereas
\(L_aR/\sqrt n\le(51/64)(1/8)=51/512<7/16\).  Consequently,
\(\alpha_k\le\alpha_f\) ensures
\begin{equation}
 \norm{\Delta_k}_F\le r_-/2,
 \label{eq:fiber-radius-app}
\end{equation}
so Lemma~\ref{lem:polar-app} applies at every iteration.  Notice also that
\(\alpha_fG<1/2\); hence the same cap guarantees the polar radius
used in Lemma~\ref{lem:tangent-perturb}.

\subsection{One-step descent}

\begin{proof}[Proof of Lemma~\ref{lem:descent}]
Lemma~\ref{lem:polar-app} and \eqref{eq:average-update-app} give the exact
identity
\begin{equation}
 \bar X_{k+1}-\bar X_k
 =\mathcal L_{\widehat X_k}[P_{\bar X_k}\Delta_k]+\mathcal R_{\widehat X_k}(\Delta_k).
 \label{eq:polar-update-app}
\end{equation}
Substitution into \eqref{eq:chord-smoothness-app} produces four terms:
\begin{align}
 \E[f(\bar X_{k+1})-f(\bar X_k)]
 &\le\E\ip{\grad f(\bar X_k)}{\mathcal L_{\widehat X_k}[P_{\bar X_k}\widehat F_k]}
 -\alpha_k\E\ip{\grad f(\bar X_k)}{\mathcal L_{\widehat X_k}[P_{\bar X_k}\widehat S_k]}\notag\\
 &\quad+\E\ip{\grad f(\bar X_k)}{\mathcal R_{\widehat X_k}(\Delta_k)}
 +\frac{L_f}{2}\E\norm{\bar X_{k+1}-\bar X_k}_F^2.
 \label{eq:four-descent-terms-app}
\end{align}
We bound these terms in their displayed order, keeping every absorption
coefficient explicit.

First decompose
\(\widehat S_k=\widehat G_k+\widehat{\boldsymbol\Xi}_k\).
All factors in the inner product containing the current noise, except the
noise itself, are \(\mathcal F_k\)-measurable.  Linearity and
\eqref{eq:noise-zero-app} therefore give
\[
 \E[\ip{\grad f(\bar X_k)}{\mathcal L_{\widehat X_k}[P_{\bar X_k}\widehat{\boldsymbol\Xi}_k]}
 \mid\mathcal F_k]=0.
\]
Since \(P_{\bar X_k}\grad f(\bar X_k)=\grad f(\bar X_k)\),
\[
 P_{\bar X_k}\widehat G_k=\grad f(\bar X_k)+P_{\bar X_k}(\widehat G_k-\grad f(\bar X_k)).
\]
Using equations \plaineqref{eq:polar-lower-app},
\plaineqref{eq:polar-upper-app}, and
\plaineqref{eq:average-gradient-gap-app}, then applying Cauchy--Schwarz,
gives
\begin{align*}
 \E\ip{\grad f(\bar X_k)}{\mathcal L_{\widehat X_k}[P_{\bar X_k}\widehat S_k]}
 &\ge\frac1{r_+}\E\norm{\grad f(\bar X_k)}_F^2
 -\frac{L_g}{r_-}
 \sqrt{\E\norm{\grad f(\bar X_k)}_F^2\frac{\E d(\Xb_k)^2}{n}}.
\end{align*}
The scalar inequality
\(ab\le a^2/(4r_+)+r_+b^2\) now yields
\begin{equation}
 -\alpha_k\E\ip{\grad f(\bar X_k)}{\mathcal L_{\widehat X_k}[P_{\bar X_k}\widehat S_k]}
 \le-\frac{3\alpha_k}{4r_+}\E\norm{\grad f(\bar X_k)}_F^2
 +\frac{r_+L_g^2\alpha_k}{r_-^2n}\E d(\Xb_k)^2.
 \label{eq:linear-descent-app}
\end{equation}

For the landing term, \eqref{eq:polar-upper-app}, Cauchy--Schwarz, and the
weaker form of \eqref{eq:tangent-landing-app} give
\begin{equation}
 \E\ip{\grad f(\bar X_k)}{\mathcal L_{\widehat X_k}[P_{\bar X_k}\widehat F_k]}
 \le\frac{\sqrt{C_a}}{r_-\sqrt n}
 \sqrt{\E d(\Xb_k)^4\,\E\norm{\grad f(\bar X_k)}_F^2}.
 \label{eq:linear-landing-app}
\end{equation}

For the polar remainder, Lemma~\ref{lem:polar-app} gives
\[
 \norm{\mathcal R_{\widehat X_k}(\Delta_k)}_F
 \le C_p\norm{P_{\bar X_k}\Delta_k}_F\norm{\Delta_k}_F.
\]
Applying \(\norm{A+B}_F^2\le2\norm A_F^2+2\norm B_F^2\) to both squared
factors and multiplying the results gives
\begin{align}
 \norm{P_{\bar X_k}\Delta_k}_F^2\norm{\Delta_k}_F^2
 \le4\{&\norm{P_{\bar X_k}\widehat F_k}_F^2\norm{\widehat F_k}_F^2
 +\alpha_k^2\norm{P_{\bar X_k}\widehat F_k}_F^2\norm{\widehat S_k}_F^2\notag\\
 &+\alpha_k^2\norm{P_{\bar X_k}\widehat S_k}_F^2\norm{\widehat F_k}_F^2
 +\alpha_k^4\norm{P_{\bar X_k}\widehat S_k}_F^2\norm{\widehat S_k}_F^2\}.
 \label{eq:four-products-app}
\end{align}
The four products in \eqref{eq:four-products-app} are bounded one by one by
equations \plaineqref{eq:full-landing-app},
\plaineqref{eq:tangent-landing-app},
\plaineqref{eq:pathwise-signal-app}, and
\plaineqref{eq:average-signal-app}.  After
taking expectation, the result is
\begin{align}
 \E[\norm{P_{\bar X_k}\Delta_k}_F^2\norm{\Delta_k}_F^2]
 &\le\frac{4C_a(G^2+L_a^2R^2)}{n}\E d(\Xb_k)^4
 +\frac{4G^2L_a^2\alpha_k^2}{n}\E d(\Xb_k)^2\notag\\
 &\quad+8G^2\alpha_k^4\E\norm{\grad f(\bar X_k)}_F^2
 +\frac{8G^2L_g^2\alpha_k^4}{n}\E d(\Xb_k)^2
 +4G^2\alpha_k^4\frac{\sigma^2}{nb}.
 \label{eq:product-bound-app}
\end{align}
For example, the last product in \eqref{eq:four-products-app} is at most
\(4G^2\alpha_k^4\norm{\widehat S_k}_F^2\); applying
\eqref{eq:average-signal-app} produces exactly the last three terms in
\eqref{eq:product-bound-app}.  The other three products give the first two
terms, after using \(d(\Xb_k)\le R\) and \(n\ge1\).

By Cauchy--Schwarz and
\(\sqrt{x_1+\cdots+x_m}\le\sum_j\sqrt{x_j}\), the first and third terms of
\eqref{eq:product-bound-app} contribute, respectively,
\begin{align}
 &\frac{2C_p\sqrt{C_a(G^2+L_a^2R^2)}}{\sqrt n}
 \sqrt{\E d(\Xb_k)^4\,\E\norm{\grad f(\bar X_k)}_F^2},
 \label{eq:remainder-mixed-app}\\
 &2\sqrt2C_pG\alpha_k^2\E\norm{\grad f(\bar X_k)}_F^2.
 \label{eq:remainder-gradient-app}
\end{align}
For the other three square-root products, Young's inequality gives
\begin{align}
 &2C_pGL_a\alpha_k
 \sqrt{\frac{\E\norm{\grad f(\bar X_k)}_F^2\,\E d(\Xb_k)^2}{n}}
 \notag\\
 &\qquad\le\frac{\alpha_k}{16r_+}\E\norm{\grad f(\bar X_k)}_F^2
 +\frac{16r_+C_p^2G^2L_a^2\alpha_k}{n}\E d(\Xb_k)^2,
 \label{eq:descent-young-template-app}\\
 &2\sqrt2C_pGL_g\alpha_k^2
 \sqrt{\frac{\E\norm{\grad f(\bar X_k)}_F^2\,\E d(\Xb_k)^2}{n}}
 \notag\\
 &\qquad\le\frac{\alpha_k}{16r_+}\E\norm{\grad f(\bar X_k)}_F^2
 +\frac{32r_+C_p^2G^2L_g^2\alpha_k^3}{n}\E d(\Xb_k)^2,
 \notag\\
 &2C_pG\alpha_k^2
 \sqrt{\frac{\sigma^2}{nb}\E\norm{\grad f(\bar X_k)}_F^2}
 \notag\\
 &\qquad\le\frac{\alpha_k}{16r_+}\E\norm{\grad f(\bar X_k)}_F^2
 +16r_+C_p^2G^2\alpha_k^3\frac{\sigma^2}{nb}.
 \notag
\end{align}
Since \(\alpha_k\le1\), the two
\(\alpha_k^3\) remainders are bounded by the required
\(\alpha_k\E d(\Xb_k)^2/n\) and
\(\alpha_k^2\sigma^2/(nb)\) terms.  These three applications spend exactly
\(3\alpha_k/(16r_+)\) of the descent coefficient.

It remains to bound the chord term. Equations \plaineqref{eq:polar-update-app},
\plaineqref{eq:polar-upper-app},
\plaineqref{eq:polar-remainder-app}, and
\plaineqref{eq:fiber-radius-app} give
\[
 \norm{\bar X_{k+1}-\bar X_k}_F
 \le \left(r_-^{-1}+\frac{C_pr_-}{2}\right)\norm{P_{\bar X_k}\Delta_k}_F.
\]
Also,
\(\norm{P_{\bar X_k}\Delta_k}_F^2\le
2\norm{P_{\bar X_k}\widehat F_k}_F^2+2\alpha_k^2\norm{P_{\bar X_k}\widehat S_k}_F^2\).
Consequently, by equations \plaineqref{eq:tangent-landing-app} and
\plaineqref{eq:average-signal-app},
\begin{align}
 \frac{L_f}{2}\E\norm{\bar X_{k+1}-\bar X_k}_F^2
 &\le L_f\left(r_-^{-1}+\frac{C_pr_-}{2}\right)^2\left\{
 \frac{C_a}{n}\E d(\Xb_k)^4
 +2\alpha_k^2\E\norm{\grad f(\bar X_k)}_F^2\right.\notag\\
 &\hspace{35mm}\left.
 +\frac{2L_g^2\alpha_k^2}{n}\E d(\Xb_k)^2
 +\alpha_k^2\frac{\sigma^2}{nb}\right\}.
 \label{eq:chord-bound-app}
\end{align}

We now collect the coefficients.  The main descent in
\eqref{eq:linear-descent-app} is \(-12\alpha_k/(16r_+)\). The three
Young inequalities above add
\(3\alpha_k/(16r_+)\).  The only remaining positive multiples of
\(\E\norm{\grad f(\bar X_k)}_F^2\) are \eqref{eq:remainder-gradient-app} and the second
term of \eqref{eq:chord-bound-app}.  Put
\begin{equation}
 C_g:=2\sqrt2C_pG+2L_f\left(r_-^{-1}+\frac{C_pr_-}{2}\right)^2,
 \qquad
 \alpha_g:=\frac1{16r_+C_g}.
 \label{eq:descent-cap-app}
\end{equation}
Then \(\alpha_k\le\alpha_g\) implies
\(C_g\alpha_k^2\le\alpha_k/(16r_+)\), leaving
\(-8\alpha_k/(16r_+)=-\alpha_k/(2r_+)\).  Thus we may take
\begin{equation}
 C_3:=\frac1{2r_+}.
 \label{eq:descent-coercivity-app}
\end{equation}
All other coefficients were displayed above. We may take
\[
 C_2:=\max\left\{
 \begin{aligned}
 &\frac{r_+L_g^2}{r_-^2}
 +16r_+C_p^2G^2(L_a^2+2L_g^2)
 +2L_f\left(r_-^{-1}+\frac{C_pr_-}{2}\right)^2L_g^2,\\
 &16r_+C_p^2G^2+L_f\left(r_-^{-1}+\frac{C_pr_-}{2}\right)^2,\quad
 L_f\left(r_-^{-1}+\frac{C_pr_-}{2}\right)^2C_a,\\
 &\frac{\sqrt{C_a}}{r_-}
 +2C_p\sqrt{C_a(G^2+L_a^2R^2)}
 \end{aligned}
 \right\}.
\]
Equations \plaineqref{eq:four-descent-terms-app}--
\plaineqref{eq:chord-bound-app} now give
\begin{align}
 \E[f(\bar X_{k+1})-f(\bar X_k)]
 &\le-C_3\alpha_k\E\norm{\grad f(\bar X_k)}_F^2
 +\frac{C_2\alpha_k}{n}\E d(\Xb_k)^2
 +\frac{C_2\alpha_k^2\sigma^2}{nb}\notag\\
 &\quad+\frac{C_2}{n}\E d(\Xb_k)^4
 +\frac{C_2}{\sqrt n}
 \sqrt{\E d(\Xb_k)^4\,\E\norm{\grad f(\bar X_k)}_F^2},
 \label{eq:descent-final-app}
\end{align}
which proves Lemma~\ref{lem:descent}.
\end{proof}

The constants \(C_2\) and \(C_3\) are fixed entirely by this one-step
argument: \(C_3\) is the descent coefficient left after the four explicit
absorptions, while \(C_2\) dominates the variance, disagreement,
fourth-order, and mixed remainders.  The stable scalar energy estimate used
in Appendix~\ref{app:accumulation} only accumulates the already established
disagreement recursion.

\section{Error accumulation and rate proofs}
\label{app:rates}
\label{app:accumulation}
\raggedbottom

All constants in this appendix depend only on the compact tube, smoothness
and oracle bounds, and the fixed network spectral gap. They are independent
of \(K,b,\sigma,\xi\). The numbered constants \(C_0,C_1,C_2,C_3\) in the main
statements denote common enlargements of the fixed constants below.

\subsection{Stable-energy weighted disagreement summation}

We use the following scalar energy estimate from
\citet[Lemma~2]{Xu2015Augmented} to accumulate the stable disagreement
recursion, omitting the unused forcing term at the final index.

\begin{lemma}[Stable scalar energy estimate]
\label{lem:stable-energy-app}
Suppose nonnegative sequences satisfy
\(v_{k+1}\le\eta v_k+w_k\), where \(0<\eta<1\).  Then, for every
\(K\ge1\),
\begin{equation}
 \left(\sum_{k=0}^{K}v_k^2\right)^{1/2}
 \le\frac{\sqrt2}{1-\eta}
 \left(\sum_{k=0}^{K-1}w_k^2\right)^{1/2}
 +\sqrt{\frac{2}{1-\eta^2}}\,v_0.
 \label{eq:stable-energy-app}
\end{equation}
\end{lemma}

Let \(\alpha_k\) be either constant or
\(\alpha_0/\sqrt{k+1}\), so that
\(0<\alpha_{k+1}\le\alpha_k\le\alpha_0\le1\).  Set
\[
 D_k:=\E d(\Xb_k)^2,
 \qquad \mathcal G_k:=\E\norm{\grad f(\bar X_k)}_F^2,
 \qquad A:=\xi^2+\frac{\sigma^2}{b}.
\]
Taking total expectation in \eqref{eq:disagreement-final-app}, multiplying
by \(\alpha_{k+1}\), taking square roots, and using
\(\sqrt{x_1+\cdots+x_m}\le\sum_j\sqrt{x_j}\) gives, for
\(v_k:=\sqrt{\alpha_kD_k}\),
\begin{align}
 v_{k+1}
 &\le\frac{\sqrt{3+\rho^2}}2\,v_k
 +\sqrt{C_1nA}\,\alpha_k^{3/2}
 +\sqrt{C_1n}\,\alpha_k^{5/2}\sqrt{\mathcal G_k}
 +\sqrt{\frac{C_1\sigma^2}{b}}\,\alpha_k^{5/2}.
 \label{eq:weighted-energy-recursion-app}
\end{align}
For example,
\(\sqrt{\alpha_{k+1}\alpha_k^2}\le\alpha_k^{3/2}\) and
\(\sqrt{\alpha_{k+1}\alpha_k^4}\le\alpha_k^{5/2}\).
Only these two index-shift inequalities~are~needed.

Apply Lemma~\ref{lem:stable-energy-app} with
\(\eta=\sqrt{3+\rho^2}/2\), and use the
triangle inequality for the three forcing sequences in
\eqref{eq:weighted-energy-recursion-app}.  Squaring the resulting four-term
bound with \((a_1+\cdots+a_4)^2\le4\sum_j a_j^2\) yields
\begin{align}
 \sum_{k=0}^{K}\alpha_kD_k
 \le C_E\!\left\{\alpha_0D_0
 +C_1nA\sum_{k<K}\alpha_k^3
 +C_1n\sum_{k<K}\alpha_k^5\mathcal G_k
 +\frac{C_1\sigma^2}{b}\sum_{k<K}\alpha_k^5\right\},
 \label{eq:weighted-disagreement-sum-app}
\end{align}
where the explicit fixed choice
\[
 C_E:=4\max\left\{
 \frac{2}{(1-\sqrt{3+\rho^2}/2)^2},\frac{8}{1-\rho^2}\right\}
\]
is independent of \(K,b,\sigma,\xi\) (and is uniform over a network family
whenever \(\rho\) is bounded uniformly below one). Finally,
\begin{align*}
 \alpha_k^5\mathcal G_k&\le\alpha_0^4\alpha_k\mathcal G_k,\\
 \frac{\sigma^2}{b}\alpha_k^5
 &\le n\left(\xi^2+\frac{\sigma^2}{b}\right)\alpha_k^3,
\end{align*}
because \(\alpha_k\le1\) and \(n\ge1\).  The prescribed initialization is
deterministic, so \(D_0=d(\Xb_0)^2\); for a random initialization the same
proof simply replaces this quantity by its expectation.  Dropping the
nonnegative \(k=K\) term on the left and taking
\(C_0\ge C_E\max\{1,C_1\}\) proves
\begin{equation}
 \sum_{k<K}\alpha_k\E d(\Xb_k)^2
 \le C_0\!\left\{
 \alpha_0d(\Xb_0)^2
 +n\left(\xi^2+\frac{\sigma^2}{b}\right)\sum_{k<K}\alpha_k^3
 +n\alpha_0^4\sum_{k<K}\alpha_k
       \E\norm{\grad f(\bar X_k)}_F^2\right\}.
 \label{eq:disagreement-accumulation-app}
\end{equation}
Conditional independence gives the
one-step \(1/(nb)\) variance bound, and the energy argument retains the powers of
\(n\), \(b\), and \(\alpha_k\) required by the master bound.  The final
\(\alpha_k^5\) noise term is deliberately upper-bounded by the displayed
\(nA\alpha_k^3\) network term.

\subsection{Higher-order accumulation}

\begin{lemma}[Higher-order accumulation]
\label{lem:higher-order-accumulation-app}
Suppose tube invariance holds and the stepsize is either constant or
\(\alpha_k=\alpha_0/\sqrt{k+1}\), with \(0<\alpha_0\le1\). There is a
fixed \(C_h>0\) such that, for every \(\varepsilon>0\) and \(K\ge1\),
\begin{align}
 &\frac{C_2}{n}\sum_{k<K}\E d(\Xb_k)^4
 +\frac{C_2}{\sqrt n}\sum_{k<K}
 \sqrt{\E d(\Xb_k)^4\,
       \E\norm{\grad f(\bar X_k)}_F^2}\notag\\
 &\quad\le C_h\left\{\frac{d(\Xb_0)^2}{n}
                    +\frac{d(\Xb_0)}{\sqrt n}\right\}
 +\varepsilon\sum_{k<K}\alpha_k
       \E\norm{\grad f(\bar X_k)}_F^2\notag\\
 &\qquad\quad+C_h\left(\alpha_0+\frac1\varepsilon\right)
       \sum_{k<K}\alpha_k\E d(\Xb_k)^2.
 \label{eq:higher-order-accumulation-app}
\end{align}
\end{lemma}

\begin{proof}
Iterating the pathwise recursion \eqref{eq:pathwise-tube-recursion-app} gives
\begin{equation}
 d(\Xb_k)
 \le\rho^kd(\Xb_0)
 +\sqrt nG\sum_{\ell=0}^{k-1}\rho^{k-1-\ell}\alpha_\ell.
 \label{eq:pathwise-unrolled-app}
\end{equation}
For a constant schedule, the convolution in
\eqref{eq:pathwise-unrolled-app} is at most
\(\alpha_k/(1-\rho)\). For
\(\alpha_k=\alpha_0/\sqrt{k+1}\), set \(j=k-1-\ell\). For
\(0\le j\le k-1\),
\[
 \frac{\alpha_{k-1-j}}{\alpha_k}
 =\sqrt{\frac{k+1}{k-j}}
 \le\sqrt{j+2},
\]
where the squared inequality follows from
\((j+2)(k-j)-(k+1)=(j+1)(k-j-1)\ge0\). Therefore both schedules obey
\begin{equation}
 d(\Xb_k)
 \le\rho^kd(\Xb_0)
 +\sqrt nG\alpha_k
   \sum_{j=0}^{\infty}\rho^j\sqrt{j+2}.
 \label{eq:pathwise-schedule-app}
\end{equation}
The series in \eqref{eq:pathwise-schedule-app} is finite by the ratio test.
Its lower bound \(1/(1-\rho)\) also \mbox{covers constant stepsizes}.

Squaring \eqref{eq:pathwise-schedule-app}, multiplying by
\(d(\Xb_k)^2\), and then using the pathwise tube bound
\(d(\Xb_k)\le R\) gives
\begin{align}
 \E d(\Xb_k)^4
 &\le 2R^2\rho^{2k}d(\Xb_0)^2\notag\\
 &\quad+2nG^2\alpha_k^2
 \left(\sum_{j=0}^{\infty}\rho^j\sqrt{j+2}\right)^2
 \E d(\Xb_k)^2.
 \label{eq:fourth-moment-reduction-app}
\end{align}
Consequently, geometric summation and
\(\alpha_k^2\le\alpha_0\alpha_k\) imply
\begin{align}
 \frac{C_2}{n}\sum_{k<K}\E d(\Xb_k)^4
 &\le \frac{2C_2R^2}{1-\rho^2}\frac{d(\Xb_0)^2}{n}\notag\\
 &\quad+2C_2G^2
 \left(\sum_{j=0}^{\infty}\rho^j\sqrt{j+2}\right)^2
 \alpha_0\sum_{k<K}\alpha_k\E d(\Xb_k)^2.
 \label{eq:fourth-sum-app}
\end{align}

Next, take the square root of \eqref{eq:fourth-moment-reduction-app}, use
\(\sqrt{x+y}\le\sqrt x+\sqrt y\) and
\(\E\norm{\grad f(\bar X_k)}_F^2\le G_f^2\), and multiply by
\(C_2/\sqrt n\). This yields, term by term,
\begin{align*}
 &\frac{C_2}{\sqrt n}
 \sqrt{\E d(\Xb_k)^4\,
       \E\norm{\grad f(\bar X_k)}_F^2}\\
 &\quad\le
 \frac{\sqrt2C_2RG_f}{\sqrt n}\rho^kd(\Xb_0)
 +\sqrt2C_2G
 \left(\sum_{j=0}^{\infty}\rho^j\sqrt{j+2}\right)
 \alpha_k
 \sqrt{\E d(\Xb_k)^2\,
       \E\norm{\grad f(\bar X_k)}_F^2}\\
 &\quad\le
 \frac{\sqrt2C_2RG_f}{\sqrt n}\rho^kd(\Xb_0)
 +\varepsilon\alpha_k\E\norm{\grad f(\bar X_k)}_F^2\\
 &\qquad+
 \frac{C_2^2G^2}{2\varepsilon}
 \left(\sum_{j=0}^{\infty}\rho^j\sqrt{j+2}\right)^2
 \alpha_k\E d(\Xb_k)^2,
\end{align*}
where the last line is Young's inequality. Summing this display and using
\(\sum_{k<K}\rho^k\le(1-\rho)^{-1}\), then combining it with
\eqref{eq:fourth-sum-app}, proves
\eqref{eq:higher-order-accumulation-app} after choosing \(C_h\) no smaller
than the finitely many displayed fixed coefficients.
\end{proof}

\subsection{One unified stepsize cap and the master bound}

Sum \eqref{eq:descent-final-app} from \(k=0\) to \(K-1\), use
\(f(\bar X_K)\ge\inf_{X\in\St}f(X)\), and apply
Lemma~\ref{lem:higher-order-accumulation-app} with
\(\varepsilon=C_3/4\). Moving this part of the gradient sum to the left
gives directly
\begin{align}
 \frac{3C_3}{4}\sum_{k<K}\alpha_k
       \E\norm{\grad f(\bar X_k)}_F^2
 &\le f(\bar X_0)-\inf_{X\in\St}f(X)
 +C_h\left\{\frac{d(\Xb_0)^2}{n}
             +\frac{d(\Xb_0)}{\sqrt n}\right\}\notag\\
 &\quad+\frac{C_2\sigma^2}{nb}\sum_{k<K}\alpha_k^2
 +\left\{C_2+C_h\left(\alpha_0+\frac4{C_3}\right)\right\}
       \sum_{k<K}\alpha_k\E d(\Xb_k)^2.
 \label{eq:master-preaccumulation-app}
\end{align}
Here \(n^{-1}\le1\) was used for the disagreement term in
\eqref{eq:descent-final-app}. Inserting
\eqref{eq:disagreement-accumulation-app} into
\eqref{eq:master-preaccumulation-app}, and using \(\alpha_0\le1\), shows
that the only term returned to the gradient sum is at most
\[
 C_0n\left\{C_2+C_h\left(1+\frac4{C_3}\right)\right\}
 \alpha_0^4\sum_{k<K}\alpha_k
 \E\norm{\grad f(\bar X_k)}_F^2.
\]
Thus every local restriction and the final feedback absorption are enforced
by the single sufficient cap
\begin{equation}
 \bar\alpha:=\min\left\{
 \begin{aligned}
 &1,\ \frac{(1-\rho)R}{2\sqrt nG},\
       \frac{r_-/2-L_aR/\sqrt n}{G},\\[-1mm]
 &\sqrt{\frac{1-\rho^2}{4C_d}},\ \frac1{16r_+C_g},\
 \left[\frac{C_3}
 {4C_0n\{C_2+C_h(1+4/C_3)\}}\right]^{1/4}
 \end{aligned}
 \right\}.
 \label{eq:unified-cap-app}
\end{equation}
All entries are positive. In order, they enforce normalization, tube
invariance, the polar-fiber radii, the disagreement margin, the linear
descent absorption, and the final error-feedback absorption.

If \(\alpha_0\le\bar\alpha\), the last entry of
\eqref{eq:unified-cap-app} makes the returned coefficient at most
\(C_3/4\). Moving it to the left of
\eqref{eq:master-preaccumulation-app} leaves at least \(C_3/2\). After
enlarging the numbered constants, we obtain
\begin{align}
 \frac{C_3}{2}\sum_{k<K}\alpha_k
       \E\norm{\grad f(\bar X_k)}_F^2
 &\le C_0\left\{f(\bar X_0)-\inf_{X\in\St}f(X)
                   +d(\Xb_0)^2+\frac{d(\Xb_0)}{\sqrt n}\right\}\notag\\
 &\quad+\frac{C_1\sigma^2}{nb}\sum_{k<K}\alpha_k^2
 +C_2n\left(\xi^2+\frac{\sigma^2}{b}\right)
       \sum_{k<K}\alpha_k^3.
 \label{eq:master-sum-app}
\end{align}
Dividing by \((C_3/2)\sum_{k<K}\alpha_k\), and using the definitions
in the main text, proves Theorem~\ref{thm:master}.

\subsection{Corollaries and the explicit network transient}

For \(\alpha_k=\alpha_0/\sqrt{k+1}\), elementary integral comparisons give
\[
 \sum_{k<K}\alpha_k\ge\alpha_0\sqrt K,
 \qquad
 \sum_{k<K}\alpha_k^2\le\alpha_0^2(1+\log K),
 \qquad
 \sum_{k<K}\alpha_k^3\le3\alpha_0^3.
\]
Substitution into \eqref{eq:master} proves
Corollary~\ref{cor:diminishing}; the minimum of nonnegative values is no
larger than their positive weighted average.

For a constant step \(\alpha\le\bar\alpha\), \eqref{eq:master} gives
\begin{equation}
 \frac1K\sum_{k<K}\E\norm{\grad f(\bar X_k)}_F^2
 \le\frac{C_0\mathcal E_0}{\alpha K}
 +\frac{C_1\alpha\sigma^2}{nb}
 +C_2n\alpha^2\left(\xi^2+\frac{\sigma^2}{b}\right).
 \label{eq:constant-step}
\end{equation}
For the horizon-dependent choice \eqref{eq:horizon-step}, the definition of
the minimum and \(\max\{x,y\}\le x+y\) give, for every \(K\ge1\),
\begin{align}
 \frac1{\alpha_KK}
 &\le\left(\frac1{\bar\alpha}+\frac1c\right)\frac1K
 +\frac{\sigma}{c\sqrt{nbK}}
 +\frac{(n\xi^2)^{1/3}}{cK^{2/3}},
 \label{eq:balanced-initial-app}\\
 \frac{\alpha_K\sigma^2}{nb}
 &\le\frac{c\sigma^2}
 {nb\{1+\sigma\sqrt{K/(nb)}\}}
 \le\frac{c\sigma}{\sqrt{nbK}},
 \label{eq:balanced-average-noise-app}\\
 n\alpha_K^2\xi^2
 &\le\frac{c^2n\xi^2}
 {\{1+(n\xi^2K)^{1/3}\}^2}
 \le\frac{c^2(n\xi^2)^{1/3}}{K^{2/3}},
 \label{eq:balanced-heterogeneity-app}\\
 \frac{n\alpha_K^2\sigma^2}{b}
 &\le\frac{c^2n\sigma^2}
 {b\{1+\sigma\sqrt{K/(nb)}\}^2}
 \le\frac{c^2n^2}{K}.
 \label{eq:network-noise-transient-app}
\end{align}
The second inequality uses \(z^2/(1+z)\le z\), the third uses
\(z^3/(1+z)^2\le z\), and the fourth uses
\(z^2/(1+z)^2\le1\), all for \(z\ge0\).  Under the fixed-network or
uniformly gapped-family conditions of Corollary~\ref{cor:speedup}, the
explicit cap \eqref{eq:unified-cap-app} gives
\(\bar\alpha^{-1}=O(\sqrt n)\).  Substitution into
\eqref{eq:constant-step}, together with \(\sqrt n/K\le n^2/K\), therefore
proves
the always-valid bound
\begin{equation}
 \frac1K\sum_{k<K}\E\norm{\grad f(\bar X_k)}_F^2
 =O\!\left(
 \frac{\sigma}{\sqrt{nbK}}
 +\frac{(n\xi^2)^{1/3}}{K^{2/3}}
 +\frac{n^2}{K}\right).
 \label{eq:speed-rate-with-transient-app}
\end{equation}
For fixed problem and network constants, the last term is no larger than the
sampling term when \(\sigma>0\) and \(K\ge n^5b/\sigma^2\), or no larger
than the heterogeneity term when \(\xi>0\) and \(K\ge n^5/\xi^2\).
Whenever either condition holds, the last term in
\eqref{eq:speed-rate-with-transient-app} is absorbed by one of the first two,
which gives the post-transient two-term form discussed after
Corollary~\ref{cor:speedup}.

\section{Additional Experimental Details and Results}
\label{app:experiments}

\subsection{Supplementary experimental settings}
\label{app:experiment-details}
\label{app:robustness}
\label{app:lrmc}

All nodes start from the same polar projection of a Gaussian matrix.
PCA gradients use the empirical covariance, and subspace distances are
computed by orthogonal Procrustes alignment \citep{Chen2021DecentralizedRG}.
The reference in Figure~\ref{fig:robustness} is the population principal
subspace; Figures~\ref{fig:scalability} and~\ref{fig:heterogeneity} use the
full empirical principal subspace. Synthetic covariance eigenvalues are
\(0.8^j\), \(j=0,\ldots,p-1\), in a random orthogonal basis.
For the network comparison, ER graphs threshold shared edge uniforms,
giving nested edge sets; graphs are generated independently of the data.

Figure~\ref{fig:robustness} and LRMC use fresh uniform minibatches,
independent across iterations and nodes, without replacement within each
batch. The heterogeneity experiment uses exact local gradients.

The LRMC scalability data distributions and dimensions are specified in
Section~\ref{subsec:lrmc}. Each stochastic gradient solves the sampled
columnwise least-squares problems. A column contributes
\((X_{\Omega_j}v_j(X)-y_{j,\Omega_j})v_j(X)^\top\) on observed rows and zero
elsewhere, averaged over the sampled columns. The reduced loss is smooth
where the observed submatrices \(X_{\Omega_j}\) have full column rank
\citep{Golub1973VariableProjection}.

Stepsize legends show coefficients before scaling. Figure~\ref{fig:robustness}
uses \(\alpha_k=\alpha_0/\sqrt{k+1}\), with no additional scaling.
Synthetic PCA, MNIST PCA, and LRMC scalability multiply \(\widehat\alpha_n\)
by \(\sqrt{n/300}\), \(\sqrt{n/300}/10000\), and \(\sqrt{n/800}\), respectively.
The LRMC coefficients \((0.2,0.3,0.4)\) give actual stepsizes
\((0.02,0.03\sqrt2,0.08)\). Heterogeneity uses the unscaled constant
\(7\times10^{-4}\). In Figure~\ref{fig:robustness}, DRSGD uses 1
communication round, consensus stepsize 1, and the polar retraction
\citep{Chen2021DecentralizedRG}. Both methods share data, initialization,
minibatch streams, and stepsizes; their execution order alternates.

Figure~\ref{fig:robustness} simulates all nodes in 1 Python process on the
server. Figures~\ref{fig:scalability} and~\ref{fig:lrmc} use MPI neighbor
communication on the same server, with 1 physical core per node and
single-threaded BLAS. Algorithm time includes sampling, gradient
evaluation, communication, and updates, and excludes initialization and
diagnostics. Figure~\ref{fig:robustness} and LRMC scalability use 20 untimed warm-up
updates on a separate state.

Figure~\ref{fig:robustness} records diagnostics every 50 updates. LRMC scalability
records the initial state, the first update, and approximately every 5
epochs thereafter. Both figures pair median cumulative time with median
error at the same recorded iteration, without interpolation. PCA
scalability uses pointwise medians on common epoch or time grids.

\subsection{Additional results on stochastic matrix completion}
\label{app:lrmc-consensus}

We compare \plados and DRSGD \citep{Chen2021DecentralizedRG} on the LRMC
problem in \eqref{eq:lrmc-objective}, with \(p=100\), \(m=1000\),
\(r=5\), and \(n=8\). Data follow the construction in
Section~\ref{subsec:lrmc}. Each entry is observed independently with
probability \(10r(p+m-r)/(pm)=0.5475\), so the expected number of
observations is 10 times the degrees of freedom of a rank-\(r\) matrix.
Each node holds 125 columns and draws fresh batches of 12 columns on a ring.
Both methods perform 3000 updates using
\(\alpha_k=n\alpha_0/\sqrt{k+1}\), with
\(\alpha_0\in\{0.04,0.06,0.08\}\); legends show \(\alpha_0\) before scaling.
DRSGD uses 1 communication round, consensus stepsize 1, and the polar
retraction. The methods share data, initialization, and minibatch streams.
All nodes are simulated in 1 process on the same local machine, with
single-threaded BLAS and the timing exclusions in Appendix~\ref{app:experiment-details}.

Figure~\ref{fig:lrmc-consensus} reports the Riemannian gradient norm and
the consensus error \(\|\Xb_k-\bar\Xb_k\|_F\), where
\(\bar\Xb_k=\one\otimes\bar X_k\). Each point pairs median cumulative
wall-clock time with median error at the same iteration over 3 runs.
Across the 3 stepsizes, \plados rapidly reduces consensus error and
maintains low disagreement as the steps diminish. Compared with DRSGD,
\plados shows a slight advantage, while both methods
reach comparable final Riemannian gradient norms. These results show stable consensus behavior for the tested
stepsizes without retractions.

\begin{figure}[htbp]
 \centering
 \includegraphics[width=0.90\linewidth]{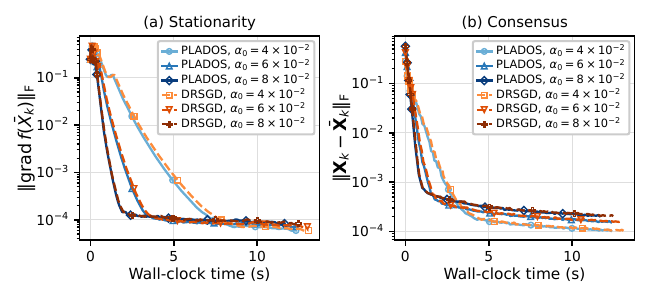}
 \caption{Stochastic matrix completion: \plados and DRSGD at 3 matched
 diminishing stepsizes. Curves show medians over 3 runs.}
 \label{fig:lrmc-consensus}
\end{figure}

\subsection{Additional scalability results}
\label{app:scalability}

Each task uses the same global dataset and initialization across node
counts within a trial. MNIST is divided by 255 without centering or
whitening, and equal shards follow a random permutation. PCA runs complete
300 epochs with batch size 1, giving the same total budget of 18 million
sample evaluations across node counts. LRMC runs complete 800 epochs
with batch size 12, giving 9.6 million sample evaluations and
\(K_n=800\cdot12000/(12n)\) updates per node. The observation mask is
also fixed across node counts.

PCA stepsizes were fixed before testing. For LRMC, the coefficients
\((0.2,0.3,0.4)\) were chosen after an initial comparison with coefficient
\(0.2\) at all node counts, so the actual step approximately doubles when
the node count doubles. The 8-node runs are reused, while the 16- and
32-node runs were performed in a later session on the same server.

For repetition \(s\), let \(t_{n,s}(k)\) denote the recorded cumulative
algorithm time and let \(\mathcal K_{n,s}\) be the saved iteration indices.
For PCA, we define the first recorded time to the common accuracy threshold by
\[
 T_{n,s}:=\min\bigl\{t_{n,s}(k): k\in\mathcal K_{n,s},\;
 d_s(\bar X_k,X^\star)\le0.01\bigr\}.
\]
For LRMC, the criterion is \(\|\grad f(\bar X_k)\|_F\le10^{-3}\).
All runs reach the corresponding threshold. The ratio \(T_{8,s}/T_{n,s}\) is computed
within each trial before taking the median and interquartile range (IQR).
There is no interpolation of hitting times; diagnostics are recorded every
epoch for PCA and approximately every 5 epochs for LRMC.
Figure~\ref{fig:paired-speedup} shows nearly linear practical scaling for
MNIST PCA and LRMC, and substantial acceleration with greater variability
for synthetic PCA. The comparisons use the stated stepsize configurations.

\begin{figure}[htbp]
 \centering
 \includegraphics[width=\linewidth]{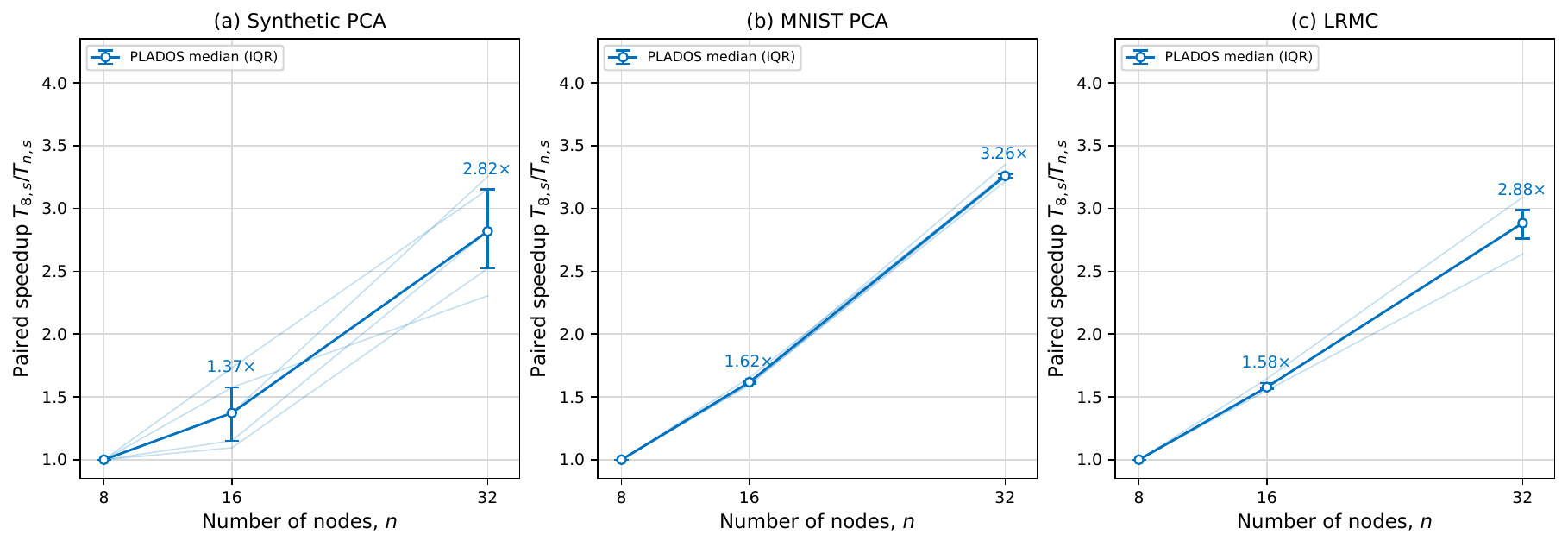}
 \caption{Paired time-to-accuracy speedup relative to 8 nodes, at
 \(d_s\le0.01\) for PCA and \(\|\grad f(\bar X_k)\|_F\le10^{-3}\) for LRMC.
 Faint lines show individual runs; thick lines and numeric
 labels show medians, and error bars show IQR.}
 \label{fig:paired-speedup}
\end{figure}

\subsection{MNIST PCA with heterogeneous partitions}
\label{app:heterogeneity}

We adapt the label-partitioned fraction of \citet{Hsieh2020NonIID} to MNIST
PCA. A fixed balanced pool contains 50000 training images, 5000 per digit.
Images are vectorized in \(\R^p\) with \(p=784\) and divided by 255 without
centering; we set \(r=5\). This pool and its global empirical reference subspace
are fixed across partitions, keeping the global optimization problem unchanged.

Images are randomly ordered within each digit, concatenated in label order,
and split into 8 blocks of 6250. Each block is shuffled once. At level
\(\rho\), node \(i\) retains its first \(\lfloor6250\rho\rfloor\) images;
the remaining images are pooled, randomly permuted, and redistributed
equally. Thus every node always has 6250 images and the global pool is
preserved. We use \(\rho\in\{0,0.2,0.4,0.6,0.8,1\}\): \(\rho=0\) gives random
assignment and \(\rho=1\) gives label-sorted blocks, with labels potentially
shared across nodes. The retained fraction \(\rho\) is a partition
parameter, not the theoretical bound \(\xi\); even a random finite
partition can have nonzero gradient heterogeneity.

Each node caches \(C_i=A_i^\top A_i/6250\) in double precision and uses the
exact empirical gradient \(-C_iX_{i,k}\), so \(\sigma=0\). The graph is an
8-node ring with Metropolis weights, \(\gamma=1/3\), and the same actual
stepsize \(\alpha=7\times10^{-4}\) for every \(\rho\). This value was chosen
in exploratory runs before the 15000-update comparison. Within each trial,
all nodes and partition levels share the same Gaussian polar
initialization. All 8 nodes are simulated in 1 process.

Each epoch consists of 1 full-gradient update. All 18 runs complete 15000 epochs,
with diagnostics recorded at every update. Figure~\ref{fig:heterogeneity}
uses pointwise medians without interpolation and the same 4000--10000
window for all curves and both metrics. Larger data heterogeneity generally
yields larger late-run gradient norms and subspace distances. The full runs
exhibit nonzero plateaus, consistent with the fixed-step residual in
\eqref{eq:constant-step}: for \(\sigma=0\),
\[
 \frac1K\sum_{k<K}\E\norm{\grad f(\bar X_k)}_F^2
 \le \frac{C_0\mathcal E_0}{\alpha K}+C_2n\alpha^2\xi^2.
\]
This bound concerns average squared stationarity. The observed plateaus
illustrate convergence to a neighborhood at the fixed stepsize.

\end{document}